\documentclass[12 pt]{amsart}
\usepackage{amsmath,mathrsfs,amsfonts,amssymb,xypic,fullpage,setspace}
\usepackage[colorlinks=true]{hyperref}
\usepackage[capitalize]{cleveref}%arxiv couldn't compile with the nameinlink option here, so I removed it
\theoremstyle{plain}
\newtheorem{thm}{Theorem}[section]
\newtheorem{lem}[thm]{Lemma}
\newtheorem{prop}[thm]{Proposition}
\crefname{prop}{Proposition}{Propositions}
\newtheorem{cor}[thm]{Corollary}
\theoremstyle{definition}
\newtheorem{defn}[thm]{Definition}
\newtheorem{rem}[thm]{Remark}
\numberwithin{equation}{section}

\numberwithin{table}{section}

\crefformat{section}{#2\S#1#3}
\Crefformat{section}{#2 Section #1#3}
\crefformat{equation}{#2(#1)#3}

\DeclareMathOperator{\Gal}{Gal}
\DeclareMathOperator{\Frob}{Frob}
\DeclareMathOperator{\GL}{GL}
\DeclareMathOperator{\Hom}{Hom}
\DeclareMathOperator{\im}{im}
\DeclareMathOperator{\ord}{ord}
\DeclareMathOperator{\ur}{ur}
\DeclareMathOperator{\Sel}{Sel}
\DeclareMathOperator{\corank}{corank}
\DeclareMathOperator{\rank}{rank}
\DeclareMathOperator{\coker}{coker}
\DeclareMathOperator{\Res}{res}
\DeclareMathOperator{\Cor}{cor}
\DeclareMathOperator{\id}{id}

\title{On the structure of Selmer groups of $p$-ordinary modular forms over $\mathbf{Z}_p$-extensions}
\author{Keenan Kidwell}
\date{\today}

\begin{document}
\maketitle

\begin{abstract}
We prove analogues of the major algebraic results of \cite{GV00} for Selmer groups of $p$-ordinary newforms over $\mathbf{Z}_p$-extensions which may be neither cyclotomic nor anticyclotomic, under a number of technical hypotheses, including a cotorsion assumption on the Selmer groups. The main complication which arises in our work is the possible presence of finite primes which can split completely in the $\mathbf{Z}_p$-extension being considered, resulting in the local cohomology groups that appear in the definition of the Selmer groups being significantly larger than they are in the case of a finitely decomposed prime. We give a careful analysis of the $\Lambda$-module structure of these local cohomology groups and identify the relevant finiteness condition one must impose to make the proof of the key cohomological surjectivity result \cite[Proposition 2.1]{GV00} work in our more general setting.
\end{abstract}%good

%\keywords{Locally analytic representation theory; $p$-adic Banach spaces; $p$-adic representation theory; principal series.}

\section{Introduction}
\label{intro}

%%****MUST FIND OUT IF NEED IRRED OR ABS IRRED FOR HOMOTHETY UNIQUENESS OF LATTICE -  SERRE'S BOOK SUGGESTS IRRED IS ENOUGH - BUT SINCE I'M VARYING COEFFICIENTS!!! 
%possibly see if I can say something about local invariants
\indent Let $p$ be a prime. The study of Selmer groups attached to arithmetic objects relative to $\mathbf{Z}_p$-extensions of number fields has its genesis in the early work of Iwasawa on class groups (and today bear's Iwasawa's name in honor of his pioneering achievements). The Selmer groups for elliptic curves were defined in Mazur's seminal paper \cite{M}, the motivation of which was in part to introduce a theory for Mordell-Weil groups analogous to Iwasawa's for class groups. Mazur used (essentially) fppf cohomology to define the Selmer groups, but his definition was translated into the language of Galois cohomology by Manin in \cite{Man71}. Using the Galois-cohomological framework, Greenberg eventually found a description of the local condition at $p$ used for defining the Selmer groups (in the case of $p$-ordinary elliptic curves) which he used in his general definition of Selmer groups for what he called $p$-ordinary Galois representations in \cite{GR89}. Since their inception, Selmer groups of this type (and generalizations) have been defined and investigated in a wide variety of settings, resulting in some spectacular applications including Kato's progress on the conjecture of Birch and Swinnerton-Dyer (via Mazur's conjecture on Selmer groups for $p$-ordinary elliptic curves in \cite{M}, generalized to $p$-ordinary newforms), to name just one example. Still, the settings in which deep understanding of Selmer groups has been attained are fairly limited, and there is indisputably much more one would like to know. In this paper, we prove, under various technical hypotheses, a number of results on the structure of Selmer groups of $p$-ordinary modular forms relative to $\mathbf{Z}_p$-extensions of number fields which may be neither cyclotomic nor anticyclotomic (these being the two classes of $\mathbf{Z}_p$-extensions for which the most is known about Selmer groups).\\%good
\indent To be somewhat more precise, we now require that $p$ be an odd prime (as we will in the main body of the paper). Let $F$ be a number field, and $F_\infty$ a $\mathbf{Z}_p$-extension of $F$ with Galois group $\Gamma=\Gal(F_\infty/F)$. We assume a pair of technical hypotheses on the ramification and decomposition of $p$ in $F$ and $F_\infty$, respectively (see \cref{notation-hyp} for the precise conditions). Let $f$ be a normalized $p$-ordinary newform of weight greater than or equal to $2$ with Hecke eigenvalues in the ring of integers $\mathscr{O}$ of a finite extension of $\mathbf{Q}_p$ with uniformizer $\pi$ and residue field $\mathbf{F}$. We follow Greenberg (\cite{GR89}) in our definition of the Selmer group $\Sel(F_\infty,A)$ for $f$ over $F_\infty$, which carries the structure of a cofinitely generated discrete $\mathscr{O}$-torsion $\Lambda$-module, where $\Lambda=\mathscr{O}[[\Gamma]]$ is the completed group ring of $\Gamma$ with coefficients in $\mathscr{O}$ and $A$ is a cofree $\mathscr{O}$-module of corank $2$ arising from a choice of $G_F$-stable $\mathscr{O}$-lattice in the $p$-adic Galois representation associated to $f$ (see \cref{adic-selmer} for the precise definition of $\Sel(F_\infty,A)$ and \cref{app-a} for the basics of such modules over the ring $\Lambda$). It is the $\Lambda$-module structure of $\Sel(F_\infty,A)$ on which we focus in this paper, guided by \cite{GV00}, which treated the case with $F=\mathbf{Q}$ and $f$ corresponding to an elliptic curve over $\mathbf{Q}$. We carry out our analysis by looking at variants of $\Sel(F_\infty,A)$ attached to an auxiliary set of primes $\Sigma_0$ (so-called non-primitive Selmer groups, see \cref{non-prim-selmer}) as well as (non-primitive) Selmer groups with coefficients in the residual Galois representation associated to $f$. The interplay between these $p$-adic and residual (primitive and non-primitive) Selmer groups can be used to obtain information about the Iwasawa invariants, $\mu(f)$ and $\lambda(f)$, of $f$, with respect to $F_\infty$. These are non-negative integers, the $\mu$ and $\lambda$-invariants of the $\Lambda$-module $\Sel(F_\infty,A)$ as defined in \cref{iwasawa-inv-defn}. One intriguing problem is to determine the degree to which these invariants are determined by the residual Galois representation of $f$. While the residual Selmer group cannot be identified with the $\pi$-torsion of the Selmer group for $f$ in general, due to the fact that the $\pi$-torsion of the Selmer group is not \emph{a priori} determined by the residual Galois representation, one can establish the desired relationship for the non-primitive analogues of these Selmer groups. Namely, if the auxiliary set $\Sigma_0$ contains the primes dividing the tame level of $f$ and the Galois representation of $f$ restricted to $G_F$ is residually absolutely irreducible, then the residual $\Sigma_0$-non-primitive Selmer group for $f$ exactly gives the $\pi$-torsion of the $\Sigma_0$-non-primitive Selmer group for $f$ (this is the content of \cref{torsion-Selmer-isom}). Under various technical assumptions on which we elaborate as they are introduced, this fact can be used to prove a numerical relationship between the $\lambda$-invariants of $p$-ordinary newforms which are congruent modulo $\pi$ in the sense that their residual Galois representations are isomorphic (\cref{main}) which is a generalization of the theorem stated in \cite[p. 237]{Gr10} (originally proved in \cite{GV00}). The other main results of this paper on the structure of Selmer groups include a surjectivity theorem for global-to-local maps in the Galois cohomology of $A$ (\cref{sel-surj}) and the non-existence of proper $\Lambda$-submodules of finite index in a sufficiently non-primitive Selmer group for $f$ (\cref{no-subs-non-prim}). Theorems of the latter type are of interest because the structure theorem for finitely generated $\Lambda$-modules (\cref{lambda-structure-thm}) only holds up to ``finite errors," which can then be shown in certain cases to in fact vanish with the help of such non-existence results. These theorems are analogues (in our more general setting) of \cite[Proposition 2.1]{GV00} and \cite[Proposition 2.5]{GV00}. The main difference between our setting and that of \cite{GV00} is that, in the $\mathbf{Z}_p$-extensions we consider, finite primes may split completely. Indeed, this work was largely motivated by a desire to understand the implications of such primes for results along the lines of those in \cite{GV00}. We therefore have included a thorough analysis of the structure of the relevant local Galois cohomology groups at these primes inspired by \cite[Lemma 3.2]{PW11}, where the formula given for the local cohomology module is not quite correct as written. Moreover, we have identified the key finiteness assumption necessary for controlling the size contributions to local cohomology made by these primes whose omission is the source of an error in \cite[Proposition A.2]{PW11} (see the discussion preceding \cref{sel-surj} for more details). \\%good
\indent As we have already indicated, similar results to those in this paper in the case $F=\mathbf{Q}$ have been proved in \cite{GV00} and \cite{Gr10}, as well as in \cite{EPW}, and analogues for $F$ an imaginary quadratic field and $F_\infty$ the anticyclotomic $\mathbf{Z}_p$-extension are obtained in \cite{PW11}. Our approach and the formulation of our results follows most closely that of the first two references, in which the $\pi$-torsion of a non-primitive Selmer group is related to a residual Selmer group. We note that, while this paper does not literally apply to the Galois representations arising from $p$-adic Tate modules of $p$-ordinary elliptic curves over the number field $F$ (since such elliptic curves do not naively correspond to modular forms such as $f$ when $F\neq\mathbf{Q}$), the methods do apply to yield the naturally analogous results, even with somewhat simplified technical hypotheses. Slightly more speculatively, the methods of this paper should apply to the Selmer groups of all $2$-dimensional potentially ordinary $p$-adic Galois representations of $G_F$ under appropriate analogues of the hypotheses we utilize. In the interest of keeping notation reasonable, and in an effort to curb the proliferation of technical conditions required, we have focused on the case of classical newforms for $\GL_2(\mathbf{Q})$.\\%good
%%change to related to whatever hypothesis when I understand it fully
\indent This work was carried out in its initial form as part of the author's Ph.D. thesis, and I would like to thank my advisor Mirela \c{C}iperiani for valuable feedback on initial drafts. I also wish to thank Robert Pollack, Olivier Fouquet, and David Loeffler for enlightening discussions related to a recurring technical hypothesis (see the discussion preceding \cref{sel-surj} for more on this hypothesis). Finally, we should point out that, after this work was completed, we learned of the paper \cite{Hach11} in which a more refined and general version of \cref{main} is obtained by methods which are quite different from those in this paper (although the tools utilized, namely local and global duality theorems for Galois cohomology, are largely the same). Regarding this paper, see also the remarks following the proof of \cref{main}.\\%good
\indent We finish this introduction with a description of the individual sections of the main body of the paper. In \cref{notation-hyp} we introduce the relevant arithmetic data and set hypotheses and notation which are in force throughout the paper. In \cref{adic-selmer} we give the detailed definition of the $p$-adic Selmer group $\Sel(F_\infty,A)$ and introduce a condition \cref{HA} which plays an essential role in all of our main results. \Cref{local-coh} contains results on the structure of the local cohomology modules which arise in the definition of the Selmer group, including the crucial description in the case of a prime which splits completely in $F_\infty$ (\cref{split-local}). Non-primitive (residual) Selmer groups are introduced in \cref{non-prim-selmer}, which also lays the groundwork for relating the Selmer groups of congruent newforms (cf. especially \cref{torsion-Selmer-isom}). In \cref{global-to-local} we prove our main surjectivity theorem (\cref{sel-surj})  for the global-to-local map of Galois cohomology whose kernel equals the Selmer group and record some useful consequences. \Cref{nice-selmer} utilizes the surjectivity theorem to deduce our second main theorem on the non-existence of proper finite index $\Lambda$-submodules of appropriate non-primitive Selmer groups (\cref{no-subs-non-prim}), as well as a corollary giving a criterion for this Selmer group to be $\mathscr{O}$-divisible. Finally, in \cref{alg-inv}, we establish our final main result comparing $\lambda$-invariants of congruent newforms (\cref{main}). We have also included two brief appendices containing an assortment of basic results on $\Lambda$-modules which are certainly well-known to (and often used by) experts but for which we could find no convenient, self-contained reference. We felt it was worthwhile to give detailed proofs of these results for immediate reference and hope that these appendices may be beneficial for someone just beginning to delve into the literature on Selmer groups over $\mathbf{Z}_p$-extensions.%good
\section{Notation and Standing Hypotheses}
\label{notation-hyp}
\indent Fix an odd prime $p$ and embeddings $\iota_p:\overline{\mathbf{Q}}\hookrightarrow\overline{\mathbf{Q}}_p$ and $\iota_\infty:\overline{\mathbf{Q}}\hookrightarrow\mathbf{C}$. Fix also a finite extension $F/\mathbf{Q}$ and a $\mathbf{Z}_p$-extension $F_\infty/F$ (for $F=\mathbf{Q}$ and, more generally, $F$ totally real, there is conjecturally only one choice for $F_\infty$, but for number fields with complex primes there are conjecturally infinitely many choices, and this is known e.g. for abelian number fields as a consequence of Leopoldt's conjecture). We write $F_n$ for the unique subextension of $F_\infty/F$ of degree $p^n$ over $F$ and $G_n$ for $\Gal(F_n/F)$. When we speak of primes of an algebraic extension of $\mathbf{Q}$, we always mean finite primes unless explicitly noted otherwise. If $L$ is an algebraic extension of $F$ in $\overline{\mathbf{Q}}$ and $\eta$ is a prime of $L$, then we write $L_\eta$ for the direct limit of the fields $L^\prime_\eta$, where $L^\prime$ is a finite subextension of $L/F$ and $L^\prime_\eta$ denotes its completion at the prime below $\eta$. We also write $I_\eta$ for the inertia group of $G_{L_\eta}$. We impose the following two conditions on the set $\Sigma_p$ of primes of $F$ above $p$:
\begin{gather}
\label{Ramp}\tag{Ram$_{p,F}$}\text{For each $\mathfrak{p}\in\Sigma_p$, $\mu_p$ is a ramified $G_{F_\mathfrak{p}}$-module, and}\\
\label{NSp}\tag{NS$_p$}\text{no prime $\mathfrak{p}\in\Sigma_p$ splits completely in $F_\infty$.}
\end{gather}
The first condition is used in the proof of one of our main results (\cref{no-subs-non-prim}) and it is not clear to us whether it can be weakened; it holds for example if the ramification indices $e(\mathfrak{p}/p)$ of primes $\mathfrak{p}\in\Sigma_p$ are all less than $p-1$, but this is not strictly necessary, since, for example, we allow the possibility that some $F_\mathfrak{p}$ is $\mathbf{Q}_p$-isomorphic to $\mathbf{Q}_p[t]/(t^p-p)$, which has absolute ramification index $p$. %make sure this is true
The second condition is essential to our method; it holds for all cyclotomic $\mathbf{Z}_p$-extensions, or for example if $\Sigma_p$ has only one element, as then this prime must ramify in $F_\infty$, and hence cannot split completely.\\%good
%is there an example where this doesn't hold?
%
\indent Having introduced our field data, we now turn to modular forms and the corresponding Galois representations. We fix a normalized newform $f=\sum_{n\geq 1}a_nq^n$ of weight $k\geq 2$, level $N$, and character $\chi$. We regard the Hecke eigenvalues $a_n$ and the character values as elements of the valuation ring of $\overline{\mathbf{Q}}_p$ via $\iota_p\circ\iota_\infty^{-1}$, and we fix as our field of coefficients a finite extension $K$ of $\mathbf{Q}_p$ with ring of integers $\mathscr{O}$, uniformizer $\pi$, and residue field $\mathbf{F}$, which contains these Hecke eigenvalues (it is then a fact that the values of $\chi$ are also in $\mathscr{O}$). When discussing Galois representations, we always write $\Frob_\ell$, $\Frob_v$, etc., to denote \emph{arithmetic} Frobenius automorphisms, and write $\epsilon:G_\mathbf{Q}\to\mathbf{Z}_p^\times$ for the $p$-adic cyclotomic character, as well as for the restriction of this character to closed subgroups corresponding to algebraic extensions of $\mathbf{Q}$ and to decomposition groups, which is a harmless abuse of notation as this restriction is then the $p$-adic cyclotomic character for the corresponding field. Let $\rho_f:G_\mathbf{Q}\rightarrow\GL_2(K)$ be the $p$-adic Galois representation associated to $f$; thus $\rho_f$ is unramified outside $pN$, and is characterized up to $\overline{\mathbf{Q}}_p$-isomorphism by the condition that for a rational prime $\ell\nmid pN$, the characteristic polynomial of $\rho_f(\Frob_\ell)$ is $X^2-a_\ell X+\chi(\ell)\ell^{k-1}$. Finally, we will denote by $\bar{\rho}_f:G_\mathbf{Q}\rightarrow\GL_2(\mathbf{F})$ the semisimple residual representation associated to $f$. Our modular form data is subject to the following assumptions:%good
\begin{gather}
\label{Ord}\tag{Ord$_p$} \text{$f$ is $p$-ordinary in the sense that $a_p$ is a $p$-adic unit,}\\
\label{Irr} \tag{Irr}\text{$\bar{\rho}_f\vert_{G_F}$ is absolutely irreducible, and}\\
\label{Rampf}\tag{Ram$_{p,f}$}\text{for each $\mathfrak{p}\in\Sigma_p$, $\overline{\rho}_f\vert_{G_F}$ is ramified at $\mathfrak{p}$.}
\end{gather}%good
The notion of ordinarity in \cref{Ord} actually depends on the choice of embedding used to regard the Fourier coefficients of $f$ as $p$-adic numbers, but as we have fixed such an embedding, this will not matter for us. The residual absolute irreducibility \cref{Irr} ensures that the integral structure on $\rho_f$ used to define the corresponding Selmer groups is unique up to homothety, even upon extending the coefficient field of the representation (this way we have a canonical Selmer group, instead of one which potentially depends upon the choice of $G_F$-stable $\mathscr{O}$-lattice in the representation space for $\rho_f$). Condition \cref{Rampf} will hold for example if $p$ is unramified in $F$, $k\not\equiv 1\pmod{p-1}$, and $\chi$ is unramified at $p$. This follows from the local structure of $\rho_f$ at $p$, which is given at the beginning \cref{adic-selmer} (the restriction of $\rho_f$ to a decomposition group at $p$ is potentially ordinary in the sense of Greenberg). This condition is used to ensure that a sufficiently non-primitive residual Selmer group for $f$ is determined up to isomorphism by $\overline{\rho}_f\vert_{G_F}$ (see \cref{max-unram} and \cref{intrinsic}), which is essential for our comparison of $\lambda$-invariants of $p$-ordinary forms which are congruent in the sense that their residual $G_F$-representations are isomorphic (\cref{main}).%good
\section{$p$-adic Selmer groups}
\label{adic-selmer}
\indent Let $V$ be a $2$-dimensional $K$-vector space with $G_\mathbf{Q}$-action via $\rho_f$, and fix a $G_F$-stable $\mathscr{O}$-lattice $T$ in $V$, setting $A=V/T$. Thus $A$ is a cofree $\mathscr{O}$-module of corank $2$ on which $G_F$ acts by $\rho_f$, and its Cartier dual $A^*=\Hom_\mathscr{O}(A,(K/\mathscr{O})(1))$ is a free $\mathscr{O}$-module of rank $2$. Since we have assumed $\bar{\rho}_f\vert_{G_F}$ to be absolutely irreducible in \cref{Irr}, the lattice $T$ is unique up to $\mathscr{O}$-scaling, and the residual representation $\bar{\rho}_f\vert_{G_F}$ is given by the action of $G_F$ on $A[\pi]\simeq T/\pi T$.\\%good
\indent Our assumption that $f$ is $p$-ordinary implies that for each prime $\mathfrak{p}\in\Sigma_p$, there is a $G_{F_\mathfrak{p}}$-stable line $V_\mathfrak{p}\subseteq V$ such that the $G_{F_\mathfrak{p}}$-action on $V_\mathfrak{p}$ is given by the product of $\epsilon^{k-1}\chi$ and an unramified character, and the $G_{F_\mathfrak{p}}$-action on $V/V_\mathfrak{p}$ is unramified \cite[\S 4.1]{EPW}. For $\mathfrak{p}\in\Sigma_p$, we set $A_\mathfrak{p}=\im(V_\mathfrak{p}\rightarrow A)$, so that $A_\mathfrak{p}$ and $A/A_\mathfrak{p}$ are both $\mathscr{O}$-cofree of corank $1$, and the action of $G_{F_\mathfrak{p}}$ on $A/A_\mathfrak{p}$ is unramified.\\%good
%
%insert specific Gr89 reference here at some point
\indent For a prime $\mathfrak{P}$ of $F_\infty$ lying over $\mathfrak{p}\in\Sigma_p$, we define the ordinary submodule $H^1_{\ord}(F_{\infty,\mathfrak{P}},A)$ of $H^1(F_{\infty,\mathfrak{P}},A)$ to be
\begin{equation*}
\ker(H^1(F_{\infty,\mathfrak{P}},A)\rightarrow H^1(I_\mathfrak{P},A/A_\mathfrak{P}))\text{,}
\end{equation*}
where $A_\mathfrak{P}$ is defined to be  $A_\mathfrak{p}$ (and so only depends on $\mathfrak{p}$). Following \cite{GR89}, we then define the Selmer group $\Sel(F_\infty,A)$ for $f$ over $F_\infty$ as the kernel of the global-to-local restriction map
\begin{equation*}
H^1(F_\infty,A)\rightarrow\prod_{\eta\nmid p}\dfrac{H^1(F_{\infty,\eta},A)}{H^1_{\ur}(F_{\infty,\eta},A)}\times\prod_{\mathfrak{P}\mid p}\dfrac{H^1(F_{\infty,\mathfrak{P}},A)}{H^1_{\ord}(F_{\infty,\mathfrak{P}},A)}\text{,}
\end{equation*}
where $\eta$ (respectively $\mathfrak{P}$) runs over the primes of $F_\infty$ not dividing (respectively dividing) $p$, and for a prime $\eta$ of the former type,
\begin{equation*}
H^1_{\ur}(F_{\infty,\eta},A)=\ker(H^1(F_{\infty,\eta},A)\to H^1(I_\eta,A))
\end{equation*}
is the submodule of unramified cohomology classes in $H^1(F_{\infty,\eta},A)$. Note that, since we have assumed $p$ is odd, the local cohomology groups for the Archimedean primes of $F$ vanish, so we may, and do, ignore them.\\%good
\indent In \cite{GR89}, in addition to the Selmer group for ordinary $p$-adic Galois representations, one can also consider the (\emph{a priori} smaller) \emph{minimal} Selmer group, requiring cocycles to be trivial away from $p$ instead of unramified (but keeping the same local conditions at primes dividing $p$). We can define the minimal Selmer group for $f$ over $F_\infty$ in the analogous way (replacing decomposition groups with inertia groups at the primes not dividing $p$). For a prime $\eta$ of $F_\infty$ lying over $v\notin\Sigma_p$, if $v$ does not split completely in $F_\infty$, then $G_{F_{\infty,\eta}}/I_\eta$ has pro-order prime to $p$, and as a result, the restriction homomorphism $H^1(F_{\infty,\eta},A)\rightarrow H^1(I_\eta,A)$ is injective. Thus, for such a prime $\eta$, the minimal local condition coincides with the unramified local condition. Therefore the Selmer group for $f$ over $F_\infty$ equals the minimal Selmer group when no prime of $F$ splits completely in $F_\infty$ (e.g. when $F_\infty$ is the cyclotomic $\mathbf{Z}_p$-extension of $F$), but these groups may differ otherwise. We do not study minimal Selmer groups in this paper (they are studied under certain conditions in \cite{PW11}), and will be content with the following result, which shows that the local conditions for the two groups can only differ at primes of $F_\infty$ lying above a prime of $F$ dividing the prime-to-$p$ part of $N$ (i.e. the tame level of $f$).%good
\begin{prop}
\label{strict-Greenberg}
If $v\nmid pN$ is a prime of $F$, then $\ker(H^1(F_v,A)\rightarrow H^1(I_v,A))=0$. Thus, if $v$ splits completely in $F_\infty$, then for any prime $\eta$ of $F_\infty$ lying over $v$, the minimal local condition and the unramified local condition at $\eta$ coincide. In particular, if every prime of $F$ dividing the level of $f$ is finitely decomposed in $F_\infty$, then the Selmer group and the minimal Selmer group coincide.
\end{prop}%good
\begin{proof}
The $\mathscr{O}$-corank of $H^1_{\ur}(F_v,A)=\ker(H^1(F_v,A)\rightarrow H^1(I_v,A))$ is the same as that of $H^0(F_v,A)$. Moreover, since $v\nmid pN$, $A$ is an unramified $G_{F_v}$-module, so the module $H^1_{\ur}(F_v,A)=A/(\Frob_v-1)A$ is $\mathscr{O}$-divisible and $H^0(F_v,A)=A^{\Frob_v=1}$. Now, if $\ell$ is the rational prime of $\mathbf{Q}$ lying below $v$, then the eigenvalues of $\Frob_\ell$ on $V$ are Weil numbers of weight $(k-1)/2$. Since $k\geq 2$, we see that, in particular, these eigenvalues are not roots of unity. The eigenvalues of $\Frob_v$ on $V$ are powers of the eigenvalues of $\Frob_\ell$ since $\rho_f(\Frob_v)$ is conjugate to a power of $\rho_f(\Frob_\ell)$. Thus $1$ is not an eigenvalue of $\Frob_v$ on $V$, so $V^{\Frob_v=1}=0$. It follows that $A^{\Frob_v=1}$ has $\mathscr{O}$-corank zero. The same is then true of $H^1_{\ur}(F_v,A)$, which is therefore $\mathscr{O}$-divisible and finite, and hence trivial, proving the first statement. In light of the discussion preceding the proposition, it follows that the only primes $w$ of $F_\infty$ where the local conditions for the Selmer group and the minimal Selmer group can differ are those lying over a prime $v$ of $F$ that divides the prime-to-$p$ part of $N$ and splits completely in $F_\infty$. So, if there are no such primes, then the Selmer group and the minimal Selmer group must coincide.
\end{proof}%good
\indent We now introduce the Iwasawa algebra $\Lambda=\mathscr{O}[[\Gamma]]$ of $\Gamma$ with coefficients in $\mathscr{O}$ (we will use this notation for the remainder of the paper). The Galois group $G_F$ acts (via conjugation) on the $\mathscr{O}$-module $H^1(F_\infty,A)$ with $G_{F_\infty}$ acting trivially, so this action allows us to regard the global cohomology group as a discrete $\mathscr{O}$-torsion $\Lambda$-module. The Selmer group $\Sel(F_\infty,A)$ is a $G_F$-stable, $\mathscr{O}$-submodule of $H^1(F_\infty,A)$, so it too is a discrete $\mathscr{O}$-torsion $\Lambda$-module. Moreover, if $\Sigma$ is a finite set of primes of $F$ containing the Archimedean primes, the primes in $\Sigma_p$, and the primes where $A$ is ramified, then we have an exact sequence
\begin{equation}
\label{sl}
0\rightarrow \Sel(F_\infty,A)\rightarrow H^1(F_\Sigma/F_\infty,A)\rightarrow\prod_{\eta\mid v\in\Sigma-\Sigma_p}\dfrac{H^1(F_{\infty,\eta},A)}{H^1_{\ur}(F_{\infty,\eta},A)}\times\prod_{\mathfrak{P}\mid p}\dfrac{H^1(F_{\infty,\mathfrak{P}},A)}{H^1_{\ord}(F_{\infty,\mathfrak{P}},A)}\text{.}
\end{equation}
According to \cite[Proposition 3]{GR89}, $H^1(F_\Sigma/F_\infty,A)$ is a cofinitely generated $\Lambda$-module, so the sequence \cref{sl} implies that $\Sel(F_\infty,A)$ is cofinitely generated as well, i.e., its $\mathscr{O}$-module Pontryagin dual $\widehat{\Sel(F_\infty,A)}=\Hom_\mathscr{O}(\Sel(F_\infty,A),K/\mathscr{O})$ is a finitely generated $\Lambda$-module (for a more thorough discussion of Pontryagin duality in the context of $\Lambda$-modules, see the beginning of \cref{app-a}). The $\Lambda$-corank (respectively the $\mu$, $\lambda$-invariant) of $\Sel(F_\infty,A)$ is the $\Lambda$-rank (respectively the $\mu$, $\lambda$-invariant) of its Pontryagin dual, as defined in \cref{iwasawa-inv-defn}. We will similarly speak of the Iwasawa invariants of any of the $\Lambda$-modules appearing in what follows. We write $\mu(f)$ and $\lambda(f)$ for the Iwasawa invariants of $\Sel(F_\infty,A)$, and also refer to them as the Iwasawa invariants of $f$ (over $F_\infty$, though we omit $F_\infty$ from the notation as it is fixed). Because we will use the condition on the set $\Sigma$ in the exact sequence \cref{sl} repeatedly, we formalize it here for convenient reference. A finite set $\Sigma$ of primes of $F$ will be said to be \emph{sufficiently large} for $A$ provided $\Sigma$ contains the Archimedean primes, the primes above $p$, and any primes where $A$ is ramified:
\begin{equation}
\label{Suff}\tag{Suff} \text{$\Sigma$ contains all $v\mid\infty$, all $\mathfrak{p}\in\Sigma_p$, and all primes where $A$ is ramified.}
\end{equation}%good
\indent All the comparison results for Iwasawa invariants of elliptic curves and modular forms which inspired this work, and accordingly our comparison result \cref{main}, will make use of the following condition:
\begin{equation}
\label{HA}
\tag{Cot}
\Sel(F_\infty,A)\text{ is cotorsion over }\Lambda\text{.}
\end{equation}
(See \cref{iwasawa-inv-defn} for the definition of a cotorsion $\Lambda$-module if it is unfamiliar.) When $f$ corresponds to an elliptic curve $E$ over $\mathbf{Q}$ with good, ordinary reduction at the primes of $\Sigma_p$ and $F_\infty$ is the cyclotomic $\mathbf{Z}_p$-extension of $F$, \cref{HA} was conjectured by Mazur in \cite{M}. For $F=\mathbf{Q}$, and more generally $F$ abelian over $\mathbf{Q}$, Mazur's conjecture follows from deep work of Kato and Rohrlich. In the case of the anticyclotomic $\mathbf{Z}_p$-extension of an imaginary quadratic $F$, with $p\geq 5$, \cref{HA} has been proved by Pollack and Weston (\cite[Theorem 1.3]{PW11}) for the Selmer groups of newforms of weight $2$ and trivial character, under some technical hypotheses on $\bar{\rho}_f$ and the factorization of the level of $f$ in $F$. It appears that little is known about the validity of \cref{HA} beyond these cases, but given what is known, it seems reasonable to expect the condition to hold in some ``sufficiently ordinary" situations beyond the cyclotomic and anticyclotomic cases. In any case, trying to compare structural invariants of Selmer groups over $\mathbf{Z}_p$-extensions for Galois representations with isomorphic residual representations in the absence of \cref{HA} will almost definitely require completely new methods, as the condition seems inextricably central to all currently known strategies for achieving such comparisons. Certainly this is true of our approach. (Even the matter of whether or not the naive definition of Iwasawa invariants is the correct one in non-cotorsion situations is somewhat unclear to the author.)\\%good
\indent In order to obtain more refined information about the structure of $\Sel(F_\infty,A)$, we need more detailed descriptions of the local cohomology groups appearing in its definition. Acquiring these descriptions is the goal of the next section.%good
\section{The $\Lambda$-module structure of local cohomology groups}
\label{local-coh}
\indent For a prime $v\nmid p$ of $F$, we define
\begin{equation*}
\mathcal{H}_v=\mathcal{H}_v(F_\infty,A)=\varinjlim_n\prod_{w\in\Sigma_{n,v}}\dfrac{H^1(F_{n,w},A)}{H^1_{\ur}(F_{n,w},A)}\text{,}
\end{equation*}
where $\Sigma_{n,v}$ is the set of primes of $F_n$ lying over $v$ and the limit is taken with respect to the restriction maps. For a prime $\mathfrak{p}$ of $\Sigma_p$, we define
\begin{equation*}
\mathcal{H}_\mathfrak{p}=\mathcal{H}_\mathfrak{p}(F_\infty,A)=\prod_{\mathfrak{P}\in\Sigma_{\infty,\mathfrak{p}}}
\dfrac{H^1(F_{\infty,\mathfrak{P}},A)}{H^1_{\ord}(F_{\infty,\mathfrak{P}},A)}\text{,}
\end{equation*}
where $\Sigma_{\infty,\mathfrak{p}}$ is the set of primes of $F_\infty$ lying over $\mathfrak{p}$, which is finite by our assumption \cref{NSp} that all such $\mathfrak{p}$ are finitely decomposed in $F_\infty$. Note that these torsion $\mathscr{O}$-modules are in fact discrete $\Lambda$-modules. For convenience of notation, we will define $\mathcal{H}_v=0$ for $v$ Archimedean, and subsequently forget about Archimedean primes. %good
\begin{prop}
\label{alt-local-cond}
For any finite set $\Sigma$ of primes of $F$ satisfying \emph{\cref{Suff}}, the sequence of $\Lambda$-modules
\begin{equation*}
0\rightarrow\Sel(F_\infty,A)\rightarrow H^1(F_\Sigma/F_\infty,A)\rightarrow\prod_{v\in\Sigma-\Sigma_p}\mathcal{H}_v\times\prod_{\mathfrak{p}\in\Sigma_p}
\mathcal{H}_\mathfrak{p}
\end{equation*}
is exact.
\end{prop}%good
\begin{proof}
Any cohomology class $\kappa\in H^1(F_\infty,A)$ arises as the restriction of a cohomology class $\kappa_n\in H^1(F_n,A)$ for some $n\geq 0$. If $\eta$ is a prime of $F_\infty$ lying over $v\notin\Sigma_p$, and $w\in\Sigma_{n,v}$, then the restriction map $H^1(F_{n,w},A)/H^1_{\ur}(F_{n,w},A)\rightarrow H^1(F_{\infty,\eta},A)/H^1_{\ur}(F_{\infty,\eta},A)$ is injective because $F_{\infty,\eta}/F_{n,w}$ is unramified. The commutative diagram
\begin{equation*}
\xymatrix{H^1(F_\infty,A) \ar[r] & \dfrac{H^1(F_{\infty,\eta},A)}{H^1_{\ur}(F_{\infty,\eta},A)}\\
H^1(F_n,A) \ar[u] \ar[r] & \dfrac{H^1(F_{n,w},A)}{H^1_{\ur}(F_{n,w},A)}\ar[u]}
\end{equation*}
of restriction maps then shows that $\kappa$ is unramified at $\eta$ if and only if $\kappa_n$ is unramified at $w$. This shows that the kernel of $H^1(F_\infty,A)\rightarrow\prod_{\eta\mid v}H^1(F_{\infty,\eta},A)/H^1_{\ur}(F_{\infty,\eta},A)$ coincides with the kernel of $H^1(F_\infty,A)\rightarrow\mathcal{H}_v$ (the latter map sends $\kappa$ to the natural image of $\kappa_n$ in $\mathcal{H}_v$). In view of the definition of $\mathcal{H}_\mathfrak{p}$ for $\mathfrak{p}\in\Sigma_p$, we conclude that $\Sel(F_\infty,A)$ is exactly the kernel in  question.
\end{proof}%good
\indent We've introduced the modules $\mathcal{H}_v$ when $v\notin\Sigma_p$ to deal with the possibility that $v$ splits completely in $F_\infty$. For such a $v$, the product of the local cohomology groups over all primes of $F_\infty$ lying over $v$ is not easily comprehended as a $\Lambda$-module (and if not zero, is probably hopelessly large). The $\Lambda$-module structure of $\mathcal{H}_v$, on the other hand, can be understood. When $v$ is finitely decomposed in $F_\infty$, $\mathcal{H}_v$ is just a product of local cohomology groups, and the structure of these groups has been determined by Greenberg.%good
\begin{prop}
\label{fd-structure}
For a prime $v\notin\Sigma_p$ of $F$, let $\Sigma_{\infty,v}$ denote the set of primes of $F_\infty$ lying above $v$.
\label{local-non-split}
\begin{enumerate}
\item For a prime $v\notin\Sigma_p$ that is finitely decomposed in $F_\infty$, we have
\begin{equation*}
\mathcal{H}_v\simeq\prod_{\eta\in\Sigma_{\infty,v}}H^1(F_{\infty,\eta},A)
\end{equation*}
as $\Lambda$-modules, and $\mathcal{H}_v$ is a cofinitely generated, cotorsion $\Lambda$-module with $\mu$-invariant zero and $\lambda$-invariant
\begin{equation*}
\sum_{\eta\in\Sigma_{\infty,v}}\corank_{\mathscr{O}}(H^1(F_{\infty,\eta},A))\text{.}
\end{equation*}
\item For a prime $\mathfrak{p}\in\Sigma_p$, $\mathcal{H}_\mathfrak{p}$ is a cofinitely generated $\Lambda$-module with $\Lambda$-corank $[F_\mathfrak{p}:\mathbf{Q}_p]$ and $\mu$-invariant zero.
\end{enumerate}
\end{prop}%good
\begin{proof}
The isomorphism for $v\notin\Sigma_p$ holds because the number of primes in $\Sigma_{n,v}$ is constant for $n$ sufficiently large (equal to the cardinality of $\Sigma_{\infty,v}$), as direct limits commute with finite products, and because $H^1_{\ur}(F_{\infty,\eta},A)=0$ for $\eta\in\Sigma_{\infty,v}$. The assertions about the $\Lambda$-module structure of the products of local cohomology groups are then given by Proposition 1 (for $\mathfrak{p}\in\Sigma_p$) and Proposition 2 (for $v\notin\Sigma_p$) of \cite{GR89}.
\end{proof}%good
\indent Now consider a prime $v\notin\Sigma_p$ that splits completely in $F_\infty$. Then we have an isomorphism $F_v\simeq F_{\infty,\eta}$ for any prime $\eta$ of $F_\infty$ lying over $v$, giving $H^1(F_v,A)\simeq H^1(F_\infty,A)$. The finiteness of $H^1(F_v,A[\pi])$ shows that $H^1(F_v,A)[\pi]$ is finite, hence that $H^1(F_v,A)$ is a cofinitely generated $\mathscr{O}$-module. In particular $H^1(F_v,A)/H^1_{\ur}(F_v,A)$ is a cofinitely generated $\mathscr{O}$-module, and by the local Euler characteristic formula, we have
\begin{equation}
\label{eulerchar}
\corank_\mathscr{O}(H^1(F_v,A)/H^1_{\ur}(F_v,A))=\rank_\mathscr{O}(H^0(F_v,A^*))\text{.}
\end{equation}
(see the beginning of \cref{adic-selmer} for the definition of the Cartier dual $A^*$).
The $\mathscr{O}$-module structure of $H^1(F_v,A)/H^1_{\ur}(F_v,A)$ completely determines the $\Lambda$-module structure of $\mathcal{H}_v$.%good
\begin{prop}
\label{split-local}
Let $v\notin\Sigma_p$ be a prime of $F$ that splits completely in $F_\infty$, and choose an isomorphism of $\mathscr{O}$-modules
\begin{equation*}
H^1(F_v,A)/H^1_{\ur}(F_v,A)\simeq(K/\mathscr{O})^r\oplus\sum_{i=1}^t\mathscr{O}/\pi^{m_i}\mathscr{O}
\end{equation*}
for some $r\geq 0$ and $m_i\geq 0$.
Then we have
\begin{equation}
\label{module-split}
\mathcal{H}_v\simeq\widehat{\Lambda}^r\oplus\sum_{i=1}^t\widehat{\Lambda/\pi^{m_i}\Lambda}
\end{equation}
as $\Lambda$-modules, so $\mathcal{H}_v$ is a cofinitely generated $\Lambda$-module with $\Lambda$-corank $\rank_\mathscr{O}(H^0(F_v,A^*))$, $\mu$-invariant $\sum_{i=1}^tm_i$, and $\lambda$-invariant zero.
\end{prop}%good
\begin{proof}
By definition, $\mathcal{H}_v=\varinjlim_n\prod_{w\in\Sigma_{n,v}}H^1(F_{n,w},A)/H^1_{\ur}(F_{n,w},A)$, with the limit taken with respect to the restriction maps. Because $v$ splits completely in $F_\infty$, and hence also in each layer of $F_\infty$, a choice of prime $w_n$ of $F_n$ lying over $v$ gives rise to an $\mathscr{O}[G_n]$-isomorphism
\begin{equation*}
\prod_{w\in\Sigma_{n,v}}H^1(F_{n,w},A)/H^1_{\ur}(F_{n,w},A)\simeq (H^1(F_v,A)/H^1_{\ur}(F_v,A))\otimes_\mathscr{O}\mathscr{O}[G_n]
\end{equation*}
(recall from \cref{notation-hyp} that $G_n=\Gal(F_n/F)$). Choosing the primes $w_n$ above $v$ compatibly as $n\rightarrow\infty$, these isomorphisms turn the transition maps defining $\mathcal{H}_v$ into the maps coming from corestriction on the right tensor factor (see \cref{app-a} for the definition of the corestriction maps between the group rings $\mathscr{O}[G_n]$). Thus we have a $\Lambda$-module isomorphism $\mathcal{H}_v\simeq\varinjlim_n (H^1(F_v,A)/H^1_{\ur}(F_v,A))\otimes_\mathscr{O}\mathscr{O}[G_n]$, and the isomorphism \cref{module-split} follows from \cref{tensor-module}. The Iwasawa invariants of $\mathcal{H}_v$ can be read off from this isomorphism, and the equality
\begin{equation*}
\corank_\Lambda(\mathcal{H}_v)=\rank_\mathscr{O}(H^0(F_v,A^*))
\end{equation*}
follows from the isomorphism and \cref{eulerchar}.
\end{proof}%good
%\indent Our analysis of the $\Lambda$-module structure of local cohomology groups in the case of primes splitting completely in $F_\infty$ was inspired by \cite[Lemma 3.2]{PW11}. Note that, in the notation of loc. cit., the formula $\mathcal{H}_\ell=H^1(K_\ell,A_f)\otimes\Lambda^\wedge$ is not correct
%mention inspiration from Pollack-Weston, maybe their error...
%
\section{Non-primitive Selmer groups}
\label{non-prim-selmer}
\indent In this section, following Greenberg-Vatsal \cite{GV00}, we introduce non-primitive Selmer groups. A non-primitive Selmer group is defined by omitting some of the local conditions at primes of $F_\infty$ not dividing $\infty$ or $p$. If we omit enough local conditions, the $\pi$-torsion of the resulting non-primitive Selmer group for $f$ can be identified with the corresponding non-primitive residual Selmer group (to be defined below). To be precise, let $\Sigma_0$ be a finite set of primes of $F$ not containing any Archimedean primes or any primes of $\Sigma_p$. The $\Sigma_0$-non-primitive Selmer group $\Sel^{\Sigma_0}(F_\infty,A)$ is then defined as the kernel of the map
\begin{equation*}
H^1(F_\infty,A)\rightarrow\prod_{\eta\mid v\notin\Sigma_0,v\nmid p}\dfrac{H^1(F_{\infty,\eta},A)}{H^1_{\ur}(F_{\infty,\eta},A)}\times\prod_{\mathfrak{P}\mid p}\dfrac{
H^1(F_{\infty,\mathfrak{P}},A)}{H^1_{\ord}(F_{\infty,\mathfrak{P}},A)}\text{.}
\end{equation*}
As there is generally no risk of confusion about $\Sigma_0$ we will sometimes refer to $\Sel^{\Sigma_0}(F_\infty,A)$ simply as the non-primitive Selmer group (for $f$ over $F_\infty$). If $\Sigma$ is a finite set of primes of $F$ satisfying \cref{Suff} which also contains $\Sigma_0$, then by \cref{alt-local-cond} together with the definitions, we have exact sequences of $\Lambda$-modules
\begin{equation*}
0\rightarrow\Sel^{\Sigma_0}(F_\infty,A)\rightarrow H^1(F_\Sigma/F_\infty,A)\rightarrow\prod_{v\in\Sigma-\Sigma_0}\mathcal{H}_v
\end{equation*}
and
\begin{equation*}
0\rightarrow\Sel(F_\infty,A)\rightarrow\Sel^{\Sigma_0}(F_\infty,A)\rightarrow\prod_{v\in\Sigma_0}\mathcal{H}_v\text{.}
\end{equation*}
We will show in \cref{sel-surj} below that, under appropriate hypotheses, these sequences are exact on the right as well. \\%good
\indent We now define a non-primitive Selmer group for the residual representation $A[\pi]$, denoted $\Sel^{\Sigma_0}(F_\infty,A[\pi])$, in a manner analogous to the non-primitive Selmer group for $A$. Its definition is designed for comparison with the $\pi$-torsion of $\Sel^{\Sigma_0}(F_\infty,A)$. For a prime $\mathfrak{P}$ of $F_\infty$ lying over some $\mathfrak{p}\in\Sigma_p$, set
\begin{equation*}
H_{\ord}^1(F_{\infty,\mathfrak{P}},A[\pi])=\ker(H^1(F_{\infty,\mathfrak{P}},A[\pi])\rightarrow H^1(I_\mathfrak{P},A[\pi]/A_\mathfrak{P}[\pi]))\text{.}
\end{equation*}
(Recall that $A_\mathfrak{P}=A_\mathfrak{p}$ was defined at the beginning of \cref{adic-selmer} using the assumption \cref{Ord} imposed on our newform $f$.)
Then we define $\Sel^{\Sigma_0}(F_\infty,A[\pi])$ as the kernel of the map
\begin{equation*}
H^1(F_\infty,A[\pi])\rightarrow\prod_{\eta\mid v\notin\Sigma_0,v\nmid p}\dfrac{H^1(F_{\infty,\eta},A[\pi])}{H^1_{\ur}(F_{\infty,\eta},A[\pi])}
\times\prod_{\mathfrak{P}\mid p}\dfrac{H^1(F_{\infty,\mathfrak{P}},A[\pi])}{H^1_{\ord}(F_{\infty,\mathfrak{P}},A[\pi])}\text{,}
\end{equation*}
where as before the subscript $\ur$ indicates the submodule of unramified cohomology classes.\\%good
\indent The next proposition will involve the space $A[\pi]_{I_\mathfrak{p}}$ of $I_\mathfrak{p}$-coinvariants of $A[\pi]$. This is the largest $\mathbf{F}[G_{F_\mathfrak{p}}]$-quotient of $A[\pi]$ on which $I_\mathfrak{p}$ acts trivially. More explicitly, it is the quotient of $A[\pi]$ by the $\mathbf{F}[G_{F_\mathfrak{p}}]$-submodule generated by elements of the form $ga-a$ for $g\in I_\mathfrak{p}$ and $a\in A[\pi]$. It is in the proof of this proposition that we use the assumption \cref{Rampf} that $A[\pi]$ is ramified at each prime $\mathfrak{p}\in\Sigma_p$. The idea for the proposition came from the discussion of the case of elliptic curves over $\mathbf{Q}$ in \cite[\S 2]{Gr10}.%good
\begin{prop}
\label{max-unram}
If $\mathfrak{p}\in\Sigma_p$, then $A[\pi]/A_\mathfrak{p}[\pi]=A[\pi]_{I_\mathfrak{p}}$.
\end{prop}%good
\begin{proof}
Since $A$ and $A_\mathfrak{p}$ are $\pi$-divisible, we have $A[\pi]/A_\mathfrak{p}[\pi]=(A/A_\mathfrak{p})[\pi]$. By the definition of $A_\mathfrak{p}$, this quotient is an unramified $\mathbf{F}[G_{F_\mathfrak{p}}]$-module, so we have a surjective homomorphism $A[\pi]_{I_\mathfrak{p}}\rightarrow A[\pi]/A_\mathfrak{p}[\pi]$. It follows that $A[\pi]_{I_\mathfrak{p}}$ is at least $1$-dimensional over $\mathbf{F}$ (since $A/A_\mathfrak{p}$ is $\mathscr{O}$-cofree of corank $1$). Because $A[\pi]$ is ramified at $\mathfrak{p}$ by assumption \cref{Rampf}, $A[\pi]_{I_\mathfrak{p}}$ cannot be $2$-dimensional. Thus the surjection $A[\pi]_{I_\mathfrak{p}}\rightarrow A[\pi]/A_\mathfrak{p}[\pi]$ is an equality.
\end{proof}%good
\begin{rem}
\label{intrinsic}
\Cref{max-unram} shows that the local conditions defining $\Sel^{\Sigma_0}(F_\infty,A[\pi])$ only depend on $A[\pi]$ as a $G_F$-module (when \emph{a priori} they depend on $A$ as a $G_F$-module because the definition of the subspace $A_\mathfrak{p}[\pi]\subseteq A[\pi]$ for a prime $\mathfrak{p}\mid p$ makes reference to the $G_F$-module structure of $A$). This is clear at the primes not dividing $p$, where the local condition is the unramified one, and \cref{max-unram} shows that at a prime $\mathfrak{P}$ lying above $\mathfrak{p}\in\Sigma_p$, the local condition is the kernel of the map 
\begin{equation*}
H^1(F_{\infty,\mathfrak{P}},A)\rightarrow H^1(I_\mathfrak{P},A[\pi]_{I_\mathfrak{p}})
\end{equation*}
induced by the quotient map $A[\pi]\rightarrow A[\pi]_{I_\mathfrak{p}}$ and restriction to $I_\mathfrak{P}$. The definition of the quotient $A[\pi]_{I_\mathfrak{p}}$ is given entirely in terms of the $G_F$-action on $A[\pi]$ (even just the $G_{F_\mathfrak{p}}$-action). This observation is crucial to our method because it shows that, if we have two modular forms satisfying the appropriate hypotheses whose residual representations are isomorphic as $G_F$-modules, then the corresponding residual Selmer groups are isomorphic.
\end{rem}%good
\indent The next proposition shows that, if $\Sigma_0$ contains the appropriate primes, then our non-primitive residual Selmer group coincides with the $\pi$-torsion in the non-primitive $p$-adic Selmer group.%good
\begin{prop}
\label{torsion-Selmer-isom}
If $\Sigma_0$ contains all the primes of $F$ dividing the tame level of $f$, then the natural map $H^1(F_\infty,A[\pi])\rightarrow H^1(F_\infty,A)$ induces an isomorphism of $\mathscr{O}$-modules 
\begin{equation*}
\Sel^{\Sigma_0}(F_\infty,A[\pi])\simeq\Sel^{\Sigma_0}(F_\infty,A)[\pi]\text{.}
\end{equation*}
\end{prop}%good
\begin{proof}
We have a commutative diagram
\begin{equation*}
\xymatrix{H^1(F_\infty,A) \ar[r] & \prod_{\eta\mid v\notin\Sigma_0,v\neq p}H^1(I_\eta,A)\times
\prod_{\mathfrak{P}\mid p} H^1(I_\mathfrak{P},A/A_\mathfrak{P})\\
H^1(F_\infty,A[\pi]) \ar[u] \ar[r] & \prod_{\eta\mid v\notin\Sigma_0,v\nmid p}H^1(I_\eta,A[\pi])
\times\prod_{\mathfrak{P}\mid p}H^1(I_\mathfrak{P},A[\pi]/A_\mathfrak{P}[\pi]) \ar[u]}
\end{equation*}
with the vertical maps coming from the respective inclusions of $\pi$-torsion $A[\pi]\hookrightarrow A$ and $A[\pi]/A_\mathfrak{p}[\pi]=(A/A_\mathfrak{p})[\pi]\hookrightarrow A/A_\mathfrak{p}$, and the horizontal maps coming from restriction. Note that the kernel of the top (respectively bottom) horizontal map is $\Sel^{\Sigma_0}(F_\infty,A)$ (respectively $\Sel^{\Sigma_0}(F_\infty,A[\pi])$). Since $A[\pi]$ is an irreducible $\mathbf{F}[G_F]$-module by assumption \cref{Irr}, $H^0(F,A[\pi])=0$, which implies that $H^0(F_\infty,A[\pi])=0$ as $F_\infty/F$ is a pro-$p$ extension. Thus the kernel of $H^1(F_\infty,A[\pi])\rightarrow H^1(F_\infty,A)$, which is a quotient of $H^0(F_\infty,A)$, is zero. So the left-hand vertical map is injective. Its image is $H^1(F_\infty,A)[\pi]$ by the long exact sequence in cohomology arising from multiplication by $\pi$ on $A$. The commutativity of the diagram shows that this vertical map takes $\Sel^{\Sigma_0}(F_\infty,A[\pi])$ into $\Sel^{\Sigma_0}(F_\infty,A)[\pi]$. To see that the image is precisely $\Sel^{\Sigma_0}(F_\infty,A)[\pi]$, it therefore suffices to prove that the right-hand vertical arrow is injective. We do this by considering each factor map on the right. First consider a prime $\eta$ of $F_\infty$ which divides $v\notin\Sigma_0$, $v\nmid p$. Since $\eta$ does not divide the level of $f$ (as $v\notin\Sigma_0$), $A$ is unramified at $\eta$, and the kernel of the map $H^1(I_\eta,A[\pi])\rightarrow H^1(I_\eta,A)$ is $A^{I_\eta}/\pi A^{I_\eta}=A/\pi A=0$ ($A$ is a divisible $\mathscr{O}$-module). Similarly, if $\mathfrak{P}$ is a prime of $F_\infty$ dividing $p$, then, since $A/A_\mathfrak{P}$ is unramified at $\mathfrak{P}$, the kernel of the map $H^1(I_\mathfrak{P},A[\pi]/A_\mathfrak{P}[\pi])\rightarrow H^1(I_\mathfrak{P},A/A_\mathfrak{P})$ is $(A/A_\mathfrak{P})/\pi(A/A_\mathfrak{P})=0$, because $A/A_\mathfrak{P}$ is a divisible $\mathscr{O}$-module.
\end{proof}%good
\indent Combining \cref{torsion-Selmer-isom} with \cref{intrinsic}, we conclude that for $\Sigma_0$ containing the primes dividing the tame level of $f$, the $\mathscr{O}$-module $\Sel^{\Sigma_0}(F_\infty,A)[\pi]$ only depends on $A[\pi]$ as an $\mathbf{F}[G_F]$-module.%good
\section{Global-to-Local Maps}
\label{global-to-local}
\indent In this section we establish the surjectivity of a global-to-local map of Galois cohomology (under appropriate hypotheses) which allows us to compare the Iwasawa invariants of the non-primitive and the primitive Selmer groups of $f$ (though we have not explicitly introduced Iwasawa invariants for non-primitive Selmer groups, they are defined as for any cofinitely generated discrete $\mathscr{O}$-torsion $\Lambda$-module according to \cref{iwasawa-inv-defn}).\\%good
\indent Recall that, by \cref{alt-local-cond}, we have an exact sequence of $\Lambda$-modules
\begin{equation}
\label{g-l}
0\rightarrow\Sel(F_\infty,A)\rightarrow H^1(F_\Sigma/F_\infty,A)\xrightarrow{\gamma}\prod_{v\in\Sigma-\Sigma_p}\mathcal{H}_v\times\prod_{\mathfrak{p}\in\Sigma_p}
\mathcal{H}_\mathfrak{p}\text{.}
\end{equation}
By \cite[Proposition 3]{GR89},
\begin{equation*}
\corank_\Lambda(H^1(F_\Sigma/F_\infty,A))\geq \sum_{v\text{ real}}d_v^-(V)+2r_2\text{,}
\end{equation*}
where the first sum is over the real primes of $F$, $d_v^-(V)$ is the dimension of the $-1$-eigenspace for a complex conjugation above $v$ acting on $V$ (the $K$-representation space of the Galois representation $\rho_f$), and $r_2$ is the number of complex primes of $F$. Because $\rho_f$ is odd, that is, the determinant of $\rho_f$ of any complex conjugation is $-1$, $d_v^-(V)=1$ for any any real prime $v$, as otherwise the determinant of $\rho_f$ of a complex conjugation would be $1$. So, letting $r_1$ denote the number of real primes of $F$, the inequality above becomes
\begin{equation*}
\corank_\Lambda(H^1(F_\Sigma/F_\infty,A))\geq r_1+2r_2=[F:\mathbf{Q}]\text{.}
\end{equation*}%good
\indent The $\Lambda$-coranks of the factors comprising the target of the map $\gamma$ of \cref{g-l} are determined by the corresponding primes as follows:
\begin{itemize}
\item if $v\notin\Sigma_p$ is finitely decomposed in $F_\infty$, then $\mathcal{H}_v$ is $\Lambda$-cotorsion;
\item if $v\notin\Sigma_p$ is split completely in $F_\infty$, then
\begin{equation*}
\corank_\Lambda(\mathcal{H}_v)=\rank_\mathscr{O}(H^0(F_v,A^*))\text{;}
\end{equation*}
\item if $\mathfrak{p}\in\Sigma_p$, then
\begin{equation*}
\corank_\Lambda(\mathcal{H}_\mathfrak{p})=[F_\mathfrak{p}:\mathbf{Q}_p]\text{.}
\end{equation*}
\end{itemize}
The first and third assertions are restatements of parts of \cref{local-non-split}, while the second is a restatement of part of \cref{split-local}. (Recall that the module $A^*$ appearing in the second assertion is the Cartier dual $\Hom_\mathscr{O}(A,(K/\mathscr{O})(1))$ of $A$.) It follows that the target of the map $\gamma$ has $\Lambda$-corank at least
\begin{equation*}
\sum_{\mathfrak{p}\in\Sigma_p}[F_\mathfrak{p}:\mathbf{Q}_p]=[F:\mathbf{Q}]\text{,}
\end{equation*}
and if $H^0(F_v,A^*)$ vanishes (equivalently, is finite) for each prime $v\in\Sigma$ that splits completely in $F_\infty$, this is exactly the $\Lambda$-corank of the target of $\gamma$.\\%good
%
%incorporate remark about PW proof and Olivier comments
%mention other results of PW not being affected?
\indent In proving the next theorem, many versions of which have appeared in the literature, we follow the proof of \cite[Proposition 1.8]{W}. In order for the technique of proof, which is originally due to Greenberg as far as we can tell, to succeed in our setting, we require the hypothesis alluded to above that $H^0(F_v,A^*)$ vanishes for each $v\in\Sigma$ which splits completely in $F_\infty$ (and this hypothesis accordingly is made in all the subsequent results which rely on \cref{sel-surj}). This hypothesis does not arise in \cite{W} because the $\mathbf{Z}_p$-extensions considered there are cyclotomic, so all finite primes are finitely decomposed. However, another version of this result is given as Proposition A.2 of \cite{PW11}, and in the generality of loc. cit., the proof sketched there (which is based on the same arguments as our proof below) requires a variant of this vanishing hypothesis in order to be complete. Indeed, in the setup of this result (as well as ours), if the relevant hypothesis fails to hold, the surjectivity of the map of Galois cohomology is impossible due to an inequality of $\Lambda$-coranks. Our proof of \cref{sel-surj} will make it apparent where the argument breaks down in the absence of this condition. With the condition though, the same argument applies to prove the more general result of \cite{PW11}. Regarding the hypothesis itself, it actually holds in the modular form setting except in a handful of certain special cases. Namely, let $\pi_f$ be the cuspidal automorphic representation associated to $f$ (to be precise, if $\pi_{u,f}$ is the unitary representation generated by the adelization of $f$, then $\pi_f=\pi\otimes\vert\det(\cdot)\vert^{-k/2}$). Then the local-global compatibility due to Deligne, Langlands, and Carayol (using the local Langlands correspondence normalized as in \cite[\S 0.5]{Car83}) shows that the hypothesis in question can fail for a prime $v$ lying over a rational prime $\ell\neq p$ only if the weight $k$ is $3$ and $\pi_{f,\ell}$ is principal series or if the weight $k$ is $4$ and $\pi_{f,\ell}$ is a twist of Steinberg. \\ %good
\indent We will use the observation that irreducibility of the $\mathbf{F}[G_F]$-module $A[\pi]$ (which is our assumption \cref{Irr}) implies that
\begin{equation*}
H^0(F_\infty,A^*\otimes_\mathscr{O}K/\mathscr{O})=0\text{.}
\end{equation*}
Indeed, the $\pi$-torsion of $H^0(F,A^*\otimes_\mathscr{O}K/\mathscr{O})$ is the space of $G_F$-equivariant elements of $\Hom_\mathbf{F}(A[\pi],\mathbf{F}(1))$, which is zero since $A[\pi]$ and $\mathbf{F}(1)$ are irreducible of different $\mathbf{F}$-dimension. Thus there are no $G_{F_\infty}$-invariants either, since $F_\infty/F$ is a pro-$p$ extension.%good
%for argument above, use that $A[\pi^n]=T/\pi^nT$ and choose a basis for T to lift maps
\begin{thm}
\label{sel-surj}
Let $\Sigma$ be a finite set of primes of $F$ satisfying \emph{\cref{Suff}}. Assume that
\begin{enumerate}
\item the condition \emph{\cref{HA}} holds, and
\item for each prime $v\in\Sigma$ that splits completely in $F_\infty$, $H^0(F_v,A^*)=0$.
\end{enumerate}
Then the sequence
\begin{equation*}
0\rightarrow\Sel(F_\infty,A)\rightarrow H^1(F_\Sigma/F_\infty,A)\xrightarrow{\gamma}\prod_{v\in\Sigma-\Sigma_p}\mathcal{H}_v\times\prod_{\mathfrak{p}\in\Sigma_p}
\mathcal{H}_\mathfrak{p}\rightarrow 0
\end{equation*}
is exact.
\end{thm}%good
\begin{proof}
We need to prove that $\gamma$ is surjective, i.e. that $\coker(\gamma)=0$. %The target of $\gamma$ has no proper finite-index $\Lambda$-submodules:
%\begin{enumerate}
%\item for primes $v\notin\Sigma_p$ that are finitely decomposed in $F_\infty$ (respectively $\mathfrak{p}\in\Sigma_p$), $\mathcal{H}_v$ (respectively $\mathcal{H}_\mathfrak{p}$) is $\mathscr{O}$-divisible because the corresponding decomposition groups in $G_{F_\infty}$ have $p$-cohomological dimension $1$, and
%\item for primes $v\in\Sigma_p$ that split in $F_\infty$, the vanishing of $H^0(F_v,A^*)$ implies by Proposition \ref{split-local} that $\mathcal{H}_v$ is dual to a finite product of $\Lambda$-modules of the form $\Lambda/\pi^m\Lambda$, and one can verify directly that such $\Lambda$-modules have no non-zero submodules.
%\end{enumerate}
%Thus it actually suffices to prove that $\coker(\gamma)$ is finite. 
We will do this by considering similar global-to-local maps at the finite levels $F_n$ of $F_\infty$ and passing to the limit. To ensure that the kernels and cokernels of the finite level global-to-local maps are finite and trivial, respectively, we will twist the Galois module structures under consideration by a character. Let $\kappa:\Gamma\simeq 1+p\mathbf{Z}_p$ be an isomorphism of topological groups, which we also regard as a character of $G_F$ and of $\Gal(F_\Sigma/F)$. If $S$ is a discrete torsion $\mathscr{O}$-module with a continuous (i.e. smooth) $\mathscr{O}$-linear $\Gal(F_\Sigma/F)$-action, then $S_t$ will denote $S\otimes_\mathscr{O}\mathscr{O}(\kappa^t)$ for $t\in\mathbf{Z}$, where $\mathscr{O}(\kappa^t)$ is a free rank-one $\mathscr{O}$-module with $\Gal(F_\Sigma/F)$-action via $\kappa^t$ (the latter notation is consistent with the notation for twisting of $\Lambda$-modules in \cref{app-b}). This is also a discrete torsion $\mathscr{O}$-module with a continuous $\mathscr{O}$-linear $\Gal(F_\Sigma/F)$-action, and if $S$ is a discrete $\mathscr{O}$-torsion $\Lambda$-module, $S_t$ is as well (with $\Lambda$-module structure induced by the diagonal $\Gamma$-action). We have $S\simeq S_t$ as $\mathscr{O}[\Gal(F_\Sigma/F_\infty)]$-modules, and if $S$ is a discrete $\mathscr{O}$-torsion $\Lambda$-module, then $S_t$ is isomorphic to the Pontryagin dual of $\widehat{S}(\kappa^{-t})$ (see the beginning of \cref{app-a} for generalities on Pontryagin duality and the paragraph above \cref{char-poly-defn} for the definition of the twist of a $\Lambda$-module by a character such as $\kappa^t$).\\%good
\indent For each $t\in\mathbf{Z}$, we define a Selmer group $\Sel(F_\infty,A_t)$ for $A_t$ as the kernel of the map
\begin{equation*}
H^1(F_\Sigma/F_\infty,A_t)\xrightarrow{\gamma_t}\prod_{v\in\Sigma-\Sigma_p}\mathcal{H}_v(F_\infty,A_t)\times\prod_{\mathfrak{p}\in\Sigma_p}\mathcal{H}_\mathfrak{p}
(F_\infty,A_t)\text{,}
\end{equation*}
where the factors of the target are defined just as for $A$, setting $A_{t,\mathfrak{p}}=(A_\mathfrak{p})_t$ for $\mathfrak{p}\in\Sigma_p$ (see the beginning of \cref{local-coh}). With this definition, $\Sel(F_\infty,A_t)\simeq\Sel(F_\infty,A)_t$ as $\Lambda$-modules, and similarly for the local cohomology modules. Because $A\simeq A_t$ as $\Gal(F_\Sigma/F_\infty)$-modules, we have $\coker(\gamma)\simeq\coker(\gamma_t)$ as $\mathscr{O}$-modules. It therefore suffices to prove that $\coker(\gamma_t)$ vanishes for some $t\in\mathbf{Z}$.\\%good
\indent We will prove vanishing of some $\coker(\gamma_t)$ by working with Selmer groups over $F_n$ for $n\geq 0$ and taking a limit. For $t\in\mathbf{Z}$, $n\geq 0$, and a prime $\mathfrak{p}$ of $F_n$ dividing $\mathfrak{p}_0\in\Sigma_p$, we set
\begin{equation*}
H^1_{\ord}(F_{n,\mathfrak{p}},A_t)=\ker(H^1(F_{n,\mathfrak{p}},A_t)\rightarrow H^1(F_{n,\mathfrak{p}},A_t/A_{t,\mathfrak{p}}))\text{,}
\end{equation*}
where $A_{t,\mathfrak{p}}=A_{t,\mathfrak{p}_0}$. Note that this is stronger than the analogous ordinary local condition over $F_\infty$ as we are using decomposition groups instead of inertia groups. But because we are using decomposition groups, we can apply Poitou-Tate global duality to each finite level Selmer group $\Sel(F_n,A_t)$, defined as the kernel of the map
\begin{equation*}
H^1(F_\Sigma/F_n,A_t)\xrightarrow{\gamma_{n,t}}\prod_{w\mid v\in\Sigma-\Sigma_p}\dfrac{H^1(F_{n,w},A_t)}{H^1_{\ur}(F_{n,w},A_t)}\times\prod_{\mathfrak{p}\mid p}
\dfrac{H^1(F_{n,\mathfrak{p}},A_t)}{H^1_{\ord}(F_{n,\mathfrak{p}},A_t)}\text{,}
\end{equation*}
where the products run over primes of $F_n$. Upon taking the direct limit of the maps $\gamma_{n,t}$ over $n\geq 0$, we get maps $\gamma_{\infty,t}$ with source $H^1(F_\Sigma/F_\infty,A_t)$, such that $\coker(\gamma_t)$ is a $\Lambda$-module quotient of $\coker(\gamma_{\infty,t})$ (it is a quotient because we used decomposition groups to define the local conditions at the primes dividing $p$ for the finite level Selmer groups). We will prove that for an appropriate choice of $t$, the $\mathscr{O}$-modules $\coker(\gamma_{n,t})$ are trivial for all $n\geq 0$. The desired result will follow from this.\\%good
\indent We will impose several conditions on the integers $t$ under consideration. Because $\Sel(F_\infty,A)$ is $\Lambda$-cotorsion by hypothesis, \cref{finite-inv} implies that for all but finitely many $t$ the $\mathscr{O}$-modules $\Sel(F_\infty,A_t)^{\Gamma_n}\simeq \Sel(F_\infty,A)_t^{\Gamma_n}$ will be finite for all $n\geq 0$. Similarly, because $H^0(F_{\infty},A)$ is $\Lambda$-cotorsion and $H^0(F_\infty,A)_t=H^0(F_\infty,A_t)$, for all but finitely many $t$, the $\mathscr{O}$-modules $H^0(F_n,A_t)=H^0(F_\infty,A_t)^{\Gamma_n}$ will be finite for all $n\geq 0$. We assume from now on that $t$ satisfies these conditions, which imply the following one:
\begin{enumerate}
\item $\ker(\gamma_{n,t})=\Sel(F_n,A_t)$ is finite for all $n\geq 0$\text{.}
\end{enumerate}
To see this, observe that the restriction map $H^1(F_n,A_t)\rightarrow H^1(F_\infty,A_t)$ takes $\ker(\gamma_{n,t})$ into $\Sel(F_\infty,A_t)^{\Gamma_n}$, which we have assumed is finite. The kernel of the restriction map is $H^1(F_\infty/F_n,H^0(F_\infty,A_t))$, which has the same $\mathscr{O}$-corank as $H^0(F_n,A_t)$, also assumed finite. So, indeed, $\ker(\gamma_{n,t})$ is finite for all $n\geq 0$. We now wish to impose three additional conditions on $t$:
\begin{enumerate}
\item[(ii)] for $n\geq 0$ and $w\in\Sigma_{n,v}$ with $v\nmid p$, $H^0(F_{n,w},A_t^*)$ is finite,
\item[(iii)] for $n\geq 0$ and $\mathfrak{p}\mid \mathfrak{p}_0\in\Sigma_p$, $H^0(F_{n,\mathfrak{p}},A_t/A_{t,\mathfrak{p}})$ and $H^0(F_{n,\mathfrak{p}},(A_t/A_{t,\mathfrak{p}})^*)$ are finite, and
\item[(iv)] for $n\geq 0$ and $\mathfrak{p}\mid \mathfrak{p}_0\in\Sigma_p$, $H^0(F_{n,\mathfrak{p}},(A_{t,\mathfrak{p}})^*)$ is finite.
\end{enumerate}
All three of these conditions will hold for all but finitely many $t$. For (iii) and (iv), this follows from \cref{finite-inv} applied to the Iwasawa algebra of the image of $G_{F_{\mathfrak{p}_0}}$ in $\Gamma$, which is non-trivial as we have assumed that no prime dividing $p$ splits completely in $F_\infty$ (we are using that the modules of coinvariants and invariants of a $\Lambda$-module that is finitely generated over $\mathscr{O}$ have the same $\mathscr{O}$-rank). Note that the conditions involving finiteness of the local invariants at each level of the Cartier dual modules are equivalent to the vanishing of the local invariants, since the Cartier dual modules are finite free over $\mathscr{O}$. That condition (ii) holds for all but finitely many $t$ follows from \cref{finite-inv} for as before, \emph{except} when $v$ is a prime that splits completely in $F_\infty$, in which case $H^0(F_{n,w},A_t^*)$ can be identified with $H^0(F_v,A^*)$, which vanishes by hypothesis (we cannot argue that $H^0(F_v,A^*)$ has to vanish for such $v$ as before because the image of $G_{F_v}$ in $\Gamma$ is trivial).\\%mostly good
\indent For $n\geq 0$ and $w\mid v\in\Sigma$ a prime of $F_n$, $v\notin\Sigma_p$, Tate local duality and the local Euler characteristic formula give
\begin{equation}
\label{r1}
\corank_\mathscr{O}\bigg(\dfrac{H^1(F_{n,w},A_t)}{H^1_{\ur}(F_{n,w},A_t)}\bigg)
=\rank_\mathscr{O}(H^0(F_{n,w},A_t^*))=0\text{,}
\end{equation}
the last equality coming from condition (ii).
Similarly, (iii) implies that $H^1(F_{n,\mathfrak{p}},A_t/A_{t,\mathfrak{p}})$ has $\mathscr{O}$-corank equal to $[F_{n,\mathfrak{p}}:\mathbf{Q}_p]$ for $n\geq 0$ and $\mathfrak{p}\mid p$ a prime of $F_n$. Condition (iv) then implies that
\begin{equation}
\label{r2}
\corank_\mathscr{O}\bigg(\dfrac{H^1(F_{n,\mathfrak{p}},A_t)}{H^1_{\ord}(F_{n,\mathfrak{p}},A_t)}\bigg)=[F_{n,\mathfrak{p}}:\mathbf{Q}_p]
\end{equation}
for all $n\geq 0$ and all such $\mathfrak{p}$ as well.\\%good
\indent A modification of the Poitou-Tate exact sequence gives the exact sequence
\begin{equation*}%would like to make this look a bit better
\xymatrix{0\ar[r]&\Sel(F_n,A_t)\ar[r]&H^1(F_\Sigma/F_n,A_t)\ar@/_/[dll]_{\gamma_{n,t}}&\\
\prod_{w\mid v\in\Sigma-\Sigma_p}\dfrac{H^1(F_{n,w},A_t)}{H^1_{\ur}(F_{n,w},A_t)}\times\prod_{\mathfrak{p}\mid p}\dfrac{H^1(F_{n,\mathfrak{p}},A_t)}{
H^1_{\ord}(F_{n,\mathfrak{p}},A_t)}\ar[r] &H_{1,n} \ar[r] &H_{2,n}\ar[r] &0}
\end{equation*}
where $H_{1,n}$ is dual to a submodule of $H^1(F_\Sigma/F_n,A_t^*)$ and $H_{2,n}$ is a submodule of $H^2(F_\Sigma/F_n,A_t)$. The global Euler characteristic formula shows that
\begin{equation}
\label{gec}
\corank_\mathscr{O}(H^1(F_\Sigma/F_n,A_t))=p^n[F:\mathbf{Q}]+\corank_\mathscr{O}(H^2(F_\Sigma/F_n,A_t))\text{.}
\end{equation}
Equations \cref{r1} and \cref{r2} give, for all $n\geq 0$,
\begin{align*}
\corank_\mathscr{O}\Bigg(\prod_{w\mid v\in\Sigma-\Sigma_p}\dfrac{H^1(F_{n,w},A_t)}{H^1_{\ur}(F_{n,w},A_t)}
\times\prod_{\mathfrak{p}\mid p}\dfrac{H^1(F_{n,\mathfrak{p}},A_t)}{H^1_{\ord}(F_{n,\mathfrak{p}},A_t)}\Bigg)&=\corank_\mathscr{O}\Bigg(\prod_{\mathfrak{p}\mid p}\dfrac{H^1(F_{n,\mathfrak{p}},A_t)}{H^1_{\ord}(F_{n,\mathfrak{p}},A_t)}\Bigg)\\
&=\sum_{\mathfrak{p}\mid p}[F_{n,\mathfrak{p}}:\mathbf{Q}_p]\\
&=\sum_{\mathfrak{p}_0\in\Sigma_p}\sum_{\mathfrak{p}\mid \mathfrak{p}_0}[F_{n,\mathfrak{p}}:F_{\mathfrak{p}_0}][F_{\mathfrak{p}_0}:\mathbf{Q}_p]\\
&=\sum_{\mathfrak{p}_0\in\Sigma_p}[F_{\mathfrak{p}_0}:\mathbf{Q}_p]\Bigg(\sum_{\mathfrak{p}\mid\mathfrak{p}_0}[F_{n,\mathfrak{p}}:F_{\mathfrak{p}_0}]\Bigg)\\
&=\sum_{\mathfrak{p}_0\in\Sigma_p}[F_{\mathfrak{p}_0}:\mathbf{Q}_p]p^n=p^n[F:\mathbf{Q}]\text{.}
\end{align*}
Therefore, since $\Sel(F_n,A_t)$ is finite for all $n\geq 0$ by (i), the exact sequence above combined with \cref{gec} gives that
\begin{equation*}
\corank_\mathscr{O}(H^1(F_\Sigma/F_n,A_t))=p^n[F:\mathbf{Q}]
\end{equation*}
and
\begin{equation*}
\corank_\mathscr{O}(H^2(F_\Sigma/F_n,A_t))=0\text{.}
\end{equation*}
In particular, $\coker(\gamma_{n,t})$ and $H_{2,n}$ are finite, and hence so is $H_{1,n}$. Moreover, the order of $\coker(\gamma_{n,t})$ is bounded above by that of $H_{1,n}$, which, being finite, is dual to a submodule of $H^1(F_\Sigma/F_n,A_t^*)[\pi^\infty]$. The $\mathscr{O}$-torsion submodule of $H^1(F_\Sigma/F_n,A_t^*)$ is a quotient of the module $H^0(F_n,A_t^*\otimes K/\mathscr{O})$, which in turn is a submodule of
\begin{equation*}
H^0(F_\infty,A_t^*\otimes_\mathscr{O}K/\mathscr{O})=H^0(F_\infty,A^*\otimes_\mathscr{O}K/\mathscr{O})\text{.}
\end{equation*}
But the latter group is trivial by the remarks preceding the theorem. Thus $\coker(\gamma_{n,t})=0$ for all $n\geq 0$.%mostly good
\end{proof}%mostly good
\begin{cor}
\label{sel-surj-2}
Let $\Sigma$ be a finite set of primes of $F$ satisfying \emph{\cref{Suff}}. Assume that
\begin{enumerate}
\item the condition \emph{\cref{HA}} holds, and
\item for each prime $v\in\Sigma$ that splits completely in $F_\infty$, $H^0(F_v,A^*)=0$.
\end{enumerate}
If $\Sigma_0$ is a subset of $\Sigma$ not containing any Archimedean primes or any primes above $p$, then $\Sel^{\Sigma_0}(F_\infty,A)$ is $\Lambda$-cotorsion and the sequences
\begin{equation*}
0\rightarrow\Sel^{\Sigma_0}(F_\infty,A)\rightarrow H^1(F_\Sigma/F_\infty,A)\rightarrow\prod_{v\in\Sigma-\Sigma_0-\Sigma_p}\mathcal{H}_v\times\prod_{\mathfrak{p}\in\Sigma_p}\mathcal{H}_\mathfrak{p}
\rightarrow 0
\end{equation*}
and
\begin{equation*}
0\rightarrow\Sel(F_\infty,A)\rightarrow\Sel^{\Sigma_0}(F_\infty,A)\rightarrow\prod_{v\in\Sigma_0}\mathcal{H}_v\rightarrow 0
\end{equation*}
are exact.
\end{cor}%good
\begin{proof}
The first sequence is exact by the definition of the non-primitive Selmer group together with \cref{sel-surj}. The second sequence is also exact by definition except for the surjectivity of the final map. The hypotheses together with \cref{local-non-split,split-local} imply that the target of that map is $\Lambda$-cotorsion. Thus the fact that $\Sel(F_\infty,A)$ is cotorsion implies the same for $\Sel^{\Sigma_0}(F_\infty,A)$. Finally, since $\Sel^{\Sigma_0}(F_\infty,A)$ is exactly the inverse image of $\prod_{v\in\Sigma_0}\mathcal{H}_v$ in $H^1(F_\Sigma/F_\infty,A)$, the second sequence is exact on the right as well.
\end{proof}%good
\begin{rem}
The most interesting case of the preceding corollary is when $\Sigma$ consists of $\Sigma_0$ together with the Archimedean primes and the primes above $p$.
\end{rem}%does this really have a point?%good
\section{Divisibility of the non-primitive Selmer group}
\label{nice-selmer}
\indent The main result in this section is that, under the hypotheses of \cref{sel-surj}, the $\Sigma_0$-non-primitive Selmer group for $\Sigma_0$ containing the primes dividing the tame level of $f$ has no proper $\Lambda$-submodules of finite index. First we deduce the corresponding result for $H^1(F_\Sigma/F_\infty,A)$, where $\Sigma$ satisfies \cref{Suff}.%good
\begin{prop}
\label{no-subs-global}
Let $\Sigma$ be a finite set of primes of $F$ satisfying \emph{\cref{Suff}}. Assume that
\begin{enumerate}
\item the condition \emph{\cref{HA}} holds, and
\item for each prime $v\in\Sigma$ that splits completely in $F_\infty$, $H^0(F_v,A^*)=0$.
\end{enumerate}
Then the $\Lambda$-corank of $H^1(F_\Sigma/F_\infty,A)$ is $[F:\mathbf{Q}]$, and $H^1(F_\Sigma/F_\infty,A)$ has no proper $\Lambda$-submodules of finite index.
\end{prop}%good
\begin{proof}
By \cref{sel-surj} and the hypothesis that $\Sel(F_\infty,A)$ is cotorsion, we have
\begin{equation*}
\corank_\Lambda(H^1(F_\Sigma/F_\infty,A))=\corank_\Lambda\Bigg(
\prod_{v\in\Sigma-\Sigma_p}\mathcal{H}_v\times\prod_{\mathfrak{p}\in\Sigma_p}\mathcal{H}_\mathfrak{p}\Bigg)\text{.}
\end{equation*}
The discussion at the beginning of \cref{global-to-local} shows that, in the presence of our hypothesis (ii),
\begin{align*}
\corank_\Lambda\Bigg(
\prod_{v\in\Sigma-\Sigma_p}\mathcal{H}_v\times\prod_{\mathfrak{p}\in\Sigma_p}\mathcal{H}_\mathfrak{p}\Bigg)&=
\corank_\Lambda\Bigg(\prod_{\mathfrak{p}\in\Sigma_p}\mathcal{H}_\mathfrak{p}\Bigg)\\
&=\sum_{\mathfrak{p}\in\Sigma_p}[F_\mathfrak{p}:\mathbf{Q}_p]=[F:\mathbf{Q}]\text{.}
\end{align*}
This proves the first assertion. Now we invoke \cite[Propositions 3-5]{GR89}. Proposition 3 implies that $H^2(F_\Sigma/F_\infty,A)$ is $\Lambda$-cotorsion, while Proposition 4 implies that $H^2(F_\Sigma/F_\infty,A)$ is $\Lambda$-cofree. Thus $H^2(F_\Sigma/F_\infty,A)=0$, and now Proposition 5 implies that $H^1(F_\Sigma/F_\infty,A)$ has no proper $\Lambda$-submodules of finite index.
\end{proof}%good
\indent The next lemma will allow us to deduce the desired property of the non-primitive Selmer group from the corresponding property of $H^1(F_\Sigma/F_\infty,A)$.%good
\begin{lem}
\label{no-subs-quot}
Let $Y$ be a finitely generated $\Lambda$-module, $Z$ a free $\Lambda$-submodule. If $Y$ contains no non-zero, finite $\Lambda$-submodules, then the same is true for $Y/Z$.
\end{lem}%good
\begin{proof}
See \cite[Lemma 2.6]{GV00}.
\end{proof}%good
\begin{lem}
\label{free-lambda}
Let $0\rightarrow M^\prime\rightarrow M\rightarrow M^{\prime\prime}\rightarrow 0$ be a short exact sequence of finitely generated $\Lambda$-modules with $M$ free over $\Lambda$ and $M^{\prime\prime}$ finitely generated and free over $\mathscr{O}$. Then $M^\prime$ is free over $\Lambda$.
\end{lem}%good
\begin{proof}
A finitely generated $\Lambda$-module is free if and only if its module of invariants vanishes and its module of coinvariants is $\mathscr{O}$-free (\cite[Proposition 5.3.19 (ii)]{NSW}). Thus it suffices to show that $(M^\prime)^{\Gamma}=0$ and that $M^\prime_{\Gamma}$ is $\mathscr{O}$-free. Applying the snake lemma to the endomorphism of the short exact sequence $0\rightarrow M^\prime\rightarrow M\rightarrow M^{\prime\prime}\rightarrow 0$ given by multiplication by $\gamma-1$, where $\gamma\in\Gamma$ is a topological generator, we get an exact sequence
\begin{equation}
\label{snake}
0\rightarrow(M^\prime)^{\Gamma}\rightarrow M^{\Gamma}\rightarrow(M^{\prime\prime})^\Gamma\rightarrow M_\Gamma^\prime
\rightarrow M_\Gamma\rightarrow M^{\prime\prime}_\Gamma\rightarrow 0\text{.}
\end{equation}
Since $M$ is $\Lambda$-free, $M^{\Gamma}=0$, and it follows that $(M^\prime)^{\Gamma}=0$. Taking $N=\im(M_{\Gamma}^\prime\rightarrow M_\Gamma)$, we deduce from \cref{snake} and the vanishing of $M^{\Gamma}$ the exact sequence
\begin{equation}
\label{sl2}
0\rightarrow (M^{\prime\prime})^{\Gamma}\rightarrow M_\Gamma^\prime\rightarrow N\rightarrow 0\text{.}
\end{equation}
Because $M^{\prime\prime}$ (respectively $M_\Gamma$) is finitely generated and $\mathscr{O}$-free, so is its submodule $(M^{\prime\prime})^{\Gamma}$ (respectively $N$). Thus the sequence of $\mathscr{O}$-modules \cref{sl2} splits, and we find that $M_\Gamma^\prime$ is $\mathscr{O}$-free, being isomorphic to a direct sum of $\mathscr{O}$-free modules.
\end{proof}%good
\indent In the proof of the next theorem, we closely follow the argument for \cite[Proposition 2.5]{GV00}.%
\begin{thm}
\label{no-subs-non-prim}
Let $\Sigma_0$ be a finite set of primes of $F$ not containing any Archimedean primes or any primes of $\Sigma_p$. Assume that
\begin{enumerate}
\item $\Sigma_0$ contains the primes dividing the tame level of $f$,
\item the condition \emph{\cref{HA}} holds, and
\item for each prime $v\in\Sigma_0$ that splits completely in $F_\infty$, $H^0(F_v,A^*)=0$.
\end{enumerate}
Then $\Sel^{\Sigma_0}(F_\infty,A)$ has no proper $\Lambda$-submodules of finite index.
\end{thm}%good
\begin{proof}
Let $\Sigma$ be the union of $\Sigma_0$, $\Sigma_p$, and the set of Archimedean primes. Then $\Sigma$ is satisfies \cref{Suff}, and the hypotheses of \cref{sel-surj-2} hold, so we have an exact sequence of $\Lambda$-modules
\begin{equation*}
0\rightarrow\Sel^{\Sigma_0}(F_\infty,A)\rightarrow H^1(F_\Sigma/F_\infty,A)\rightarrow\prod_{\mathfrak{p}\in\Sigma_p}\mathcal{H}_\mathfrak{p}\rightarrow 0\text{.}
\end{equation*}
Since $H^1(F_\Sigma/F_\infty,A)$ has no proper $\Lambda$-submodules of finite index by \cref{no-subs-global}, if we can prove that $\prod_{\mathfrak{p}\in\Sigma_p}\mathcal{H}_\mathfrak{p}$ is $\Lambda$-cofree, the result will follow from \cref{no-subs-quot}, wherein we take $Y=\widehat{H^1(F_\Sigma/F_\infty,A)}$, $Z=\prod_{\mathfrak{p}\in\Sigma_p}\widehat{\mathcal{H}_\mathfrak{p}}$, and $Y/Z\simeq\widehat{\Sel^{\Sigma_0}(F_\infty,f)}$. Following the proof of \cite[Proposition 2.5]{GV00}, we will prove that for each $\mathfrak{p}\in\Sigma_p$,
\begin{equation*}
\mathcal{H}_\mathfrak{p}=\prod_{\mathfrak{P}\mid\mathfrak{p}}H^1(F_{\infty,\mathfrak{P}},A)/H^1_{\ord}(F_{\infty,\mathfrak{P}},A)
\end{equation*}
is $\Lambda$-cofree.\\%good
\indent Fix $\mathfrak{p}\in\Sigma_p$ and let $D=A/A_\mathfrak{p}$. We first prove that $H^1(F_\mathfrak{p},D)$ is $\mathscr{O}$-cofree. The long exact cohomology sequence associated to map given by multiplication by $\pi$ on $D$ yields an injection
\begin{equation}
\label{piin}
H^1(F_\mathfrak{p},D)/\pi H^1(F_\mathfrak{p},D)\hookrightarrow H^2(F_\mathfrak{p},D[\pi])\text{.}
\end{equation}
The target of \cref{piin} is Cartier dual (as a finite $p$-group) to $H^0(F_\mathfrak{p},\Hom(D[\pi],\mu_p))$. If we had a non-zero, hence surjective $G_{F_\mathfrak{p}}$-equivariant homomorphism $\varphi:D[\pi]\rightarrow \mu_p$, then because $D[\pi]$ is unramified at $\mathfrak{p}$, $\mu_p$ would be unramified at $\mathfrak{p}$ as well. But this contradicts our assumption \cref{Ramp} that $\mu_p$ is a ramified $G_{F_\mathfrak{p}}$-module. So the module of $G_{F_\mathfrak{p}}$-invariants of $\Hom(D[\pi],\mu_p)$ must vanish, and thus $H^2(F_\mathfrak{p},D[\pi])=0$. By \cref{piin}, we infer that $H^1(F_\mathfrak{p},D)$ is $\mathscr{O}$-divisible. Since it is cofinitely generated as an $\mathscr{O}$-module, it is then $\mathscr{O}$-cofree, as desired. As for the corank of this $\mathscr{O}$-module, by the local Euler characteristic formula, we have
\begin{align*}
\corank_{\mathscr{O}}(H^1(F_\mathfrak{p},D))&=[F_\mathfrak{p}:\mathbf{Q}_p]+\corank_\mathscr{O}(H^0(F_\mathfrak{p},D))+\corank_\mathscr{O}(H^2(F_\mathfrak{p},D))\\
&=[F_\mathfrak{p}:\mathbf{Q}_p]
+\corank_\mathscr{O}(H^0(F_\mathfrak{p},D))\text{,}
\end{align*} 
where the last equality holds because the vanishing of $H^2(F_\mathfrak{p},D[\pi])$ implies that of $H^2(F_\mathfrak{p},D)$ (the $\pi$-torsion of the latter is a homomorphic image of the former).\\%good
\indent Now fix a prime $\mathfrak{P}$ of $F_\infty$ lying over $\mathfrak{p}$, let $\Gamma_\mathfrak{p}\subseteq\Gamma$ be the decomposition group for $\mathfrak{p}$ (which is non-trivial because $\mathfrak{p}$ does not split completely in $F_\infty$ by \cref{NSp}), and let $\Lambda_\mathfrak{p}=\mathscr{O}[[\Gamma_\mathfrak{p}]]$ be the Iwaswawa algebra of $\Gamma_\mathfrak{p}$. Because $\Gamma_\mathfrak{p}$ has cohomological dimension $1$, we have an inflation-restriction sequence
\begin{equation*}
0\rightarrow H^1(F_{\infty,\mathfrak{P}}/F_\mathfrak{p},D^{G_{F_{\infty,\mathfrak{P}}}})\rightarrow H^1(F_\mathfrak{p},D)
\rightarrow H^1(F_{\infty,\mathfrak{P}},D)^{\Gamma_\mathfrak{p}}\rightarrow 0\text{.}
\end{equation*}
The $\mathscr{O}$-corank of $H^1(F_{\infty,\mathfrak{P}}/F_\mathfrak{p},D^{G_{F_{\infty,\mathfrak{P}}}})$ is the same as the $\mathscr{O}$-corank of $H^0(F_\mathfrak{p},D)$, and from this and the computation of $\corank_\mathscr{O}(H^1(F_\mathfrak{p},D))$ above, it follows that $H^1(F_{\infty,\mathfrak{P}},D)^{\Gamma_\mathfrak{p}}$ is $\mathscr{O}$-cofree of corank $[F_\mathfrak{p}:\mathbf{Q}_p]$ (that it is $\mathscr{O}$-cofree follows from the fact that it is a quotient of $H^1(F_\mathfrak{p},D)$). By \cite[Proposition 1]{GR89}, the $\Lambda_\mathfrak{p}$-corank of $H^1(F_{\infty,\mathfrak{P}},D)$ is also $[F_\mathfrak{p}:\mathbf{Q}_p]$. An application of Nakayama's lemma now shows that $H^1(F_{\infty,\mathfrak{P}},D)$ is $\Lambda_\mathfrak{p}$-cofree of corank $[F_\mathfrak{p}:\mathbf{Q}_p]$.\\%good
\indent As $G_{F_{\infty,\mathfrak{P}}}$ has $p$-cohomological dimension $1$ (\cite[pg. 433]{Gr01}), the long exact cohomology sequence arising from the quotient map $A\to D$ implies that the homomorphism $H^1(F_{\infty,\mathfrak{P}},A)\rightarrow H^1(F_{\infty,\mathfrak{P}},D)$ is surjective. We therefore have
\begin{equation*}
\mathcal{H}_\mathfrak{P}=H^1(F_{\infty,\mathfrak{P}},A)/H^1_{\ord}(F_{\infty,\mathfrak{P}},A)\simeq\im(H^1(F_{\infty,\mathfrak{P}},D)\rightarrow H^1(I_\mathfrak{P},D))\text{.}
\end{equation*}
Since $D$ is unramified at $\mathfrak{p}$, the kernel of the restriction map to $I_\mathfrak{P}$ is $H^1(G_{F_{\infty,\mathfrak{P}}}/I_\mathfrak{P},D)$. If $\mathfrak{p}$ is unramified in $F_\infty$, then $G_{F_{\infty,\mathfrak{P}}}/I_\mathfrak{P}$ has pro-order prime to $p$, so the restriction map is injective, from which it follows that
\begin{equation*}
\mathcal{H}_\mathfrak{P}\simeq H^1(F_{\infty,\mathfrak{P}},D)\text{,}
\end{equation*}
so $\mathcal{H}_\mathfrak{P}$ is $\Lambda_\mathfrak{p}$-cofree because $H^1(F_{\infty,\mathfrak{P}},D)$ is. If instead $\mathfrak{p}$ is ramified in $F_\infty$, then $G_{F_{\infty,\mathfrak{P}}}/I_\mathfrak{P}$ is isomorphic to $\widehat{\mathbf{Z}}$, and so $H^1(G_{F_{\infty,\mathfrak{P}}}/I_\mathfrak{P},D)$ is a quotient of $D$, hence $\mathscr{O}$-cofree. Thus we have an exact sequence of finitely generated $\Lambda_\mathfrak{p}$-modules
\begin{equation*}
0\rightarrow\widehat{\mathcal{H}_\mathfrak{P}}\rightarrow\widehat{H^1(F_{\infty,\mathfrak{P}},D)}\rightarrow
\widehat{H^1(G_{F_{\infty,\mathfrak{P}}}/I_\mathfrak{P},D)}\rightarrow 0
\end{equation*}
satisfying the hypotheses of \cref{free-lambda}, which therefore implies that $\mathcal{H}_\mathfrak{P}$ is $\Lambda_\mathfrak{p}$-cofree. Thus, in either case, $\mathcal{H}_\mathfrak{P}$ is $\Lambda_\mathfrak{p}$-cofree.\\%good
\indent Finally we explain why $\mathcal{H}_\mathfrak{p}$ is cofree over $\Lambda$. The choice of a prime $\mathfrak{P}$ of $F_\infty$ above $\mathfrak{p}$ gives rise to an isomorphism of $\Lambda$-modules $\widehat{\mathcal{H}_\mathfrak{p}}\simeq\widehat{\mathcal{H}_\mathfrak{P}}\otimes_{\Lambda_\mathfrak{p}}\Lambda$. Since we have proved that $\mathcal{H}_\mathfrak{P}$ is $\Lambda_\mathfrak{p}$-cofree, we conclude that $\mathcal{H}_\mathfrak{p}$ is $\Lambda$-cofree.%good
\end{proof}
\begin{rem}
The proof of \cref{no-subs-non-prim} is the only instance where we make use of assumption \cref{Ramp}, but we do not see a way to remove this assumption.
\end{rem}%good
\begin{cor}
\label{divisible-Selmer}
Let $\Sigma_0$ be a finite set of primes of $F$ not containing any Archimedean primes or any primes of $\Sigma_p$. Assume that
\begin{enumerate}
\item $\Sigma_0$ contains the primes dividing the tame level of $f$,
\item the condition \emph{\cref{HA}} holds, and
\item for each prime $v\in\Sigma_0$ that splits completely in $F_\infty$, $H^0(F_v,A^*)=0$.
\end{enumerate}
Then the $\mu$-invariant of $\Sel^{\Sigma_0}(F_\infty,A)$ vanishes if and only if $\Sel^{\Sigma_0}(F_\infty,A[\pi])$ is finite, in which case $\Sel^{\Sigma_0}(F_\infty,A)$ is $\mathscr{O}$-divisible with
\begin{equation*}
\lambda(\Sel^{\Sigma_0}(F_\infty,A))=\dim_{\mathbf{F}}(\Sel^{\Sigma_0}(F_\infty,A[\pi]))\text{.}
\end{equation*}
\end{cor}%good
\begin{proof}
By \cref{torsion-Selmer-isom}, $\Sel^{\Sigma_0}(F_\infty,A[\pi])\simeq\Sel^{\Sigma_0}(F_\infty,A)[\pi]$ as $\mathscr{O}$-modules. The structure theorem for $\Lambda$-modules now implies that the finiteness of the residual Selmer group is equivalent to the vanishing of the $\mu$-invariant of $\Sel^{\Sigma_0}(F_\infty,A)$, and that when this happens, $\Sel^{\Sigma_0}(F_\infty,A)$ is a cofinitely generated $\mathscr{O}$-module. By \cref{no-subs-non-prim}, the $\mathscr{O}$-torsion submodule of the Pontryagin dual of $\Sel^{\Sigma_0}(F_\infty,A)$ must vanish. Thus $\Sel^{\Sigma_0}(F_\infty,A)$ is $\mathscr{O}$-cofree of corank equal to its $\lambda$-invariant, which now visibly coincides with
\begin{equation*}
\dim_\mathbf{F}(\Sel^{\Sigma_0}(F_\infty,A)[\pi])=\dim_\mathbf{F}(\Sel^{\Sigma_0}(F_\infty,A[\pi]))\text{.}
\end{equation*}
\end{proof}%good
\section{Algebraic $\lambda$-invariants}
\label{alg-inv}
\indent In this final section we use the structural results for Selmer groups established in the preceding sections to prove a result on the behavior of $\lambda$-invariants under congruences. In order to state the precise result we require some notation. Let $f_1,f_2$ be $p$-ordinary newforms (i.e. satisfying \cref{Ord}) of weight greater than or equal to $2$ (not necessarily the same weight), and tame levels $N_1$ and $N_2$, and assume that the Hecke eigenvalues of $f_1$ and $f_2$ are contained in $K$. We assume moreover that the $2$-dimensional Galois representations associated to $f_1$ and $f_2$ satisfy the conditions \cref{Irr} and \cref{Rampf}, i.e., that the residual representations on $G_F$ are absolutely irreducible and ramified at each $\mathfrak{p}\in\Sigma_p$. Choose $G_F$-stable lattices $T_1,T_2$ in the associated $K$-representation spaces of $f_1,f_2$ and let $A_1,A_2$ be the resulting discrete $\mathscr{O}$-torsion $G_F$-modules. Let $\Sel(F_\infty,A_1)$ and $\Sel(F_\infty,A_2)$ denote the Selmer groups for $f_1$ and $f_2$ over $F_\infty$ as defined in \cref{adic-selmer}, with the corresponding Iwasawa invariants denoted $\mu(f_1),\lambda(f_1)$ and $\mu(f_2),\lambda(f_2)$.\\%good
\indent Let $\Sigma_0$ be the set of primes of $F$ dividing $N=N_1N_2$, and let $\Sigma$ consist of the primes in $\Sigma_0$ together with the primes of $F$ dividing $\infty$ or $p$. We may write $N\mathscr{O}_F=(N\mathscr{O}_F)^f(N\mathscr{O}_F)^s$, where $(N\mathscr{O}_F)^f$ is divisible only by primes that are finitely decomposed in $F_\infty$ and $(N\mathscr{O}_F)^s$ is divisible only by primes that split completely in $F_\infty$. For a prime $v\notin\Sigma_p$ (respectively $\mathfrak{p}\in\Sigma_p$), denote by $\mathcal{H}_{v,i}$ (respectively $\mathcal{H}_{\mathfrak{p},i}$) the analogue for $A_i$ of the $\Lambda$-module $\mathcal{H}_v$ (respectively $\mathcal{H}_\mathfrak{p}$) defined in terms of local cohomology groups in the beginning of \cref{local-coh}. By \cref{local-non-split} (i), if $v\mid (N\mathscr{O}_F)^f$, $\mathcal{H}_{v,i}$ is a cotorsion $\Lambda$-module; let $\lambda_{v,i}$ be its $\lambda$-invariant (which is simply its $\mathscr{O}$-corank since it has $\mu$-invariant zero by loc. cit.). Finally, let the $\lambda$-invariant of $\Sel^{\Sigma_0}(F_\infty,A_i)$ be denoted by $\lambda(\Sigma_0,f_i)$ for $i=1,2$. %good
\begin{thm}
\label{main}
For $i=1,2$, assume that $\mathcal{H}_{v,i}=0$ if $v\mid(N\mathscr{O}_F)^s$. Suppose $A_1[\pi]\simeq A_2[\pi]$ as $\mathbf{F}[G_F]$-modules. Then $\Sel(F_\infty,A_1)$ is $\Lambda$-cotorsion with $\mu(f_1)=0$ if and only if $\Sel(F_\infty,A_2)$ is $\Lambda$-cotorsion with $\mu(f_2)=0$. In this case, we have
\begin{equation*}
\lambda(f_1)-\lambda(f_2)=\sum_{v\mid (N\mathscr{O}_F)^f}\lambda_{v,2}-\lambda_{v,1}\text{.}
\end{equation*}
\end{thm}%good
\begin{proof}
First note that the hypothesis on the vanishing of the modules $\mathcal{H}_{v,i}$ for $v\mid(N\mathscr{O}_F)^s$ implies that $H^0(F_v,A_i^*)=0$ for $i=1,2$ and $v\mid(N\mathscr{O}_F)^s$ (because the $\mathscr{O}$-corank of $H^0(F_v,A_i^*)$ is the $\Lambda$-corank of $\mathcal{H}_{v,i}$, by \cref{split-local}). Suppose that $\Sel(F_\infty,A_1)$ is $\Lambda$-cotorsion with $\mu(f_1)=0$. Then the hypotheses of \cref{sel-surj-2} are satisfied for $A_1$ with our choices of $\Sigma_0$ and $\Sigma$, and we therefore have an exact sequence of $\Lambda$-modules
\begin{equation}
\label{key-1}
0\rightarrow\Sel(F_\infty,A_1)\rightarrow\Sel^{\Sigma_0}(F_\infty,A_1)\rightarrow\prod_{v\mid(N\mathscr{O}_F)^f}\mathcal{H}_{v,1}\rightarrow 0\text{,}
\end{equation}
taking into account the assumption that $\mathcal{H}_{v,1}=0$ for $v\mid(N\mathscr{O}_F)^s$. The target of the surjective map in \cref{key-1} is $\Lambda$-cotorsion with $\mu$-invariant zero by \cref{local-non-split}, and as we have assumed the same for $\Sel(F_\infty,A_1)$, we conclude that $\Sel^{\Sigma_0}(F_\infty,A_1)$ is also $\Lambda$-cotorsion with $\mu$-invariant zero. \cref{divisible-Selmer} now implies that $\Sel^{\Sigma_0}(F_\infty,A_1)$ is $\mathscr{O}$-divisible with $\Sel^{\Sigma_0}(F_\infty,A_1[\pi])$ finite of $\mathbf{F}$-dimension equal to the $\lambda$-invariant $\lambda(\Sigma_0,f_1)$ of $\Sel^{\Sigma_0}(F_\infty,A_1)$. \\%good
\indent By \cref{intrinsic}, the non-primitive residual Selmer groups $\Sel^{\Sigma_0}(F_\infty,A_i[\pi])$, $i=1,2$, are determined up to $\mathscr{O}$-module isomorphism by the $\mathbf{F}[G_F]$-module structures of $A_1[\pi]$ and $A_2[\pi]$, respectively. Since we have assumed that these $\mathbf{F}[G_F]$-modules are isomorphic, it therefore follows that we have $\Lambda$-module isomorphisms
\begin{equation}
\label{key-isom}
\Sel^{\Sigma_0}(F_\infty,A_1)[\pi]\simeq\Sel^{\Sigma_0}(F_\infty,A_1[\pi])
\simeq\Sel^{\Sigma_0}(F_\infty,A_2[\pi])\simeq\Sel^{\Sigma_0}(F_\infty,A_2)[\pi]\text{,}
\end{equation}
where the first and last isomorphisms come from \cref{torsion-Selmer-isom}. In particular, because $\Sel^{\Sigma_0}(F_\infty,A_1[\pi])$ is finite, the same is true of $\Sel^{\Sigma_0}(F_\infty,A_2)[\pi]$. This implies that $\Sel^{\Sigma_0}(F_\infty,A_2)$ is $\Lambda$-cotorsion with $\mu$-invariant equal to $0$, and since $\Sel(F_\infty,A_2)\subseteq\Sel^{\Sigma_0}(F_\infty,A_2)$, the same is true of $\Sel(F_\infty,A_2)$. The hypotheses of \cref{sel-surj-2} are therefore satisfied for $A_2$, so we have an exact sequence of cotorsion $\Lambda$-modules
\begin{equation}
\label{key-2}
0\rightarrow\Sel(F_\infty,A_2)\rightarrow\Sel^{\Sigma_0}(F_\infty,A_2)\rightarrow\prod_{v\mid(N\mathscr{O}_F)^f}\mathcal{H}_{v,2}\rightarrow 0\text{.}
\end{equation}
The additivity of $\lambda$-invariants in short exact sequences of cotorsion $\Lambda$-modules applied to the sequences \cref{key-1} and \cref{key-2} gives
\begin{equation}
\label{key-3}
\lambda(f_1)+\sum_{v\mid(N\mathscr{O}_F)^f}\lambda_{v,1}=\lambda(\Sigma_0,f_1)
\end{equation}
and
\begin{equation}
\label{key-4}
\lambda(f_2)+\sum_{v\mid(N\mathscr{O}_F)^f}\lambda_{v,2}=\lambda(\Sigma_0,f_2)\text{.}
\end{equation}
The isomorphism \cref{key-isom} together with \cref{divisible-Selmer} gives $\lambda(\Sigma_0,f_1)=\lambda(\Sigma_0,f_2)$. Thus the right-hand sides of \cref{key-3} and \cref{key-4} are equal, so upon equating the left-hand sides and rearranging, we obtain
\begin{equation*}
\lambda(f_1)-\lambda(f_2)=\sum_{v\mid(N\mathscr{O}_F)^f}\lambda_{v,2}-\lambda_{v,1}\text{,}
\end{equation*}
as desired. 
\end{proof}%good
\indent The essential content of the theorem is that, under the various hypotheses, the $\lambda$-invariants of the respective modular forms depend only on the residual representations along with data coming from local cohomology at primes dividing the product of the tame levels. This result is formally and thematically similar to \cite[Theorem 4.3.3, 4.3.4]{EPW}, \cite[Theorem 3.1, 3.2]{W}, and \cite[Theorem 7.1]{PW11}, which apply in the cases $F=\mathbf{Q}$, $F_\infty$ a cyclotomic $\mathbf{Z}_p$-extension, and $F_\infty$ the anticyclotomic $\mathbf{Z}_p$-extension of an imaginary quadratic $F$, respectively (the results in these references are stated in a somewhat different form from ours, using the framework of Hida families or Galois deformations, and while \cite{EPW} and \cite{PW11} consider representations associated to modular forms as we do, \cite{W} considers higher-dimensional ``nearly ordinary" Galois representations). The theorem is a direct generalization of (and was motivated by) \cite[p. 237]{Gr10}, which addresses the case $F=\mathbf{Q}$ and $f$ corresponding to an elliptic curve, modulo the caveat that Greenberg assumes the $p$-torsion representations of the elliptic curves with which he works are merely irreducible instead of absolutely so. But this is to be expected, as for elliptic curves, one has a canonical choice of lattice in the corresponding $p$-adic Galois representation furnished by the Tate module, whereas an (essentially) unique integral structure in the general modular form case is only guaranteed under the hypothesis of residual absolute irreducibility. As mentioned in \cref{intro}, the arguments in this paper work in the elliptic curve case (i.e. for the $p$-adic Tate modules of $p$-ordinary elliptic curves defined over the base field $F$) under the weaker residual irreducibility hypothesis, with some of the technicalities simplified. \\%good
\indent The hypothesis that the modules $\mathcal{H}_{v,i}$ vanish for those $v$ which split completely in $F_\infty$, which is of course vacuous if all primes dividing $N$ are finitely decomposed in $F_\infty$, seems (unfortunately) unavoidable in our approach to \cref{main} (note that the assumption used repeatedly in earlier results of the vanishing of $H^0(F_v,A_i^*)$ for $v$ splitting completely in $F_\infty$ is equivalent to the weaker condition that the corresponding $\mathcal{H}_{v,i}$ be $\Lambda$-cotorsion by \cref{split-local}). The point is that our comparison method, based on that of \cite{GV00}, breaks down in the presence of a positive $\mu$-invariant for the Selmer groups under consideration. The modules $\mathcal{H}_{v,i}$, if finite but non-zero, \emph{will} have positive $\mu$-invariant by \cref{split-local}, which, given the second exact sequence of \cref{sel-surj-2}, would force $\Sel^{\Sigma_0}(F_\infty,A)$ to have positive $\mu$-invariant. When this occurs, we cannot naively read off the $\lambda$-invariant of this non-primitive Selmer group from its $\pi$-torsion (which is in fact infinite). As indicated in \cite[Theorem 1.1]{PW11}, it is possible (at least in the anticyclotomic setting, in contrast to what is expected for the cyclotomic setting) for the Selmer group to have a positive $\mu$-invariant, which appears to come directly from the local cohomology at primes which are split completely in $F_\infty$. The theorem suggests that, in this setting, to compare invariants, one should instead look at the aforementioned minimal Selmer group, and indeed, \cref{strict-Greenberg} shows that in a precise sense such primes which divide the relevant level account completely for the difference between the minimal Selmer group and Greenberg's (which we have used). We emphasize, however, that \cref{main} is the only instance where we are forced to use this hypothesis. It seems very likely that one could work systematically with minimal Selmer groups using the methods of this paper and obtain an analogue of \cref{main} without the hypothesis, but we have not written down the details. As we learned after the completion of this work, a comparison theorem for $\lambda$-invariants accommodating positive $\mu$-invariants has been established in \cite{Hach11} by initially working with a stricter variant of minimal Selmer groups, but the method of loc. cit. is quite different from ours and depends on higher congruences (i.e. isomorphisms of residual representations modulo possibly larger powers of $\pi$).%good
\appendix
\section{Iwasawa invariants of direct limits of discrete $\Lambda$-modules}
\label{app-a}
\indent In this appendix we establish some results about the structure of direct limits of certain discrete $\Lambda$-modules and their Pontryagin duals.\\%good
\indent Let $\Gamma$ be a free pro-$p$ group of rank one (i.e. $\Gamma$ is topologically isomorphic to the additive group $\mathbf{Z}_p$), $\mathscr{O}$ the ring of integers in a finite extension $K$ of $\mathbf{Q}_p$ with uniformizer $\pi$, and let $\Lambda=\mathscr{O}[[\Gamma]]$ be the completed group ring of $\Gamma$ with coefficients in $\mathscr{O}$ (\cite[Definition 5.2.1]{NSW}). Unlike in the main body of the paper, we do not require that $p$ be odd. The Pontryagin dual of a discrete $\mathscr{O}$-torsion $\Lambda$-module $M$ is $\widehat{M}=\Hom_\mathscr{O}(M,K/\mathscr{O})$, equipped with the compact-open topology (taking the discrete topology on $K/\mathscr{O}$); $\widehat{M}$ is then a profinite topological $\mathscr{O}$-module with $\Lambda$-module structure determined by letting any $\gamma\in\Gamma$ act on a homomorphism $\chi:M\to K/\mathscr{O}$ by $(\gamma\chi)(m)=\chi(\gamma^{-1}m)$. This gives an $\mathscr{O}[\Gamma]$-module structure which extends uniquely to a topological $\Lambda$-module structure by continuity. (Alternatively, one can simply take the $\Lambda$-action defined by $(\lambda\chi)(m)=\chi(\lambda m)$ and twist it by the automorphism of $\Lambda$ induced by inversion on $\Gamma$.) Recall that such a $\Lambda$-module $M$ is \emph{cofinitely generated} if $\widehat{M}$ is a finitely generated $\Lambda$-module. In a similar manner (taking \emph{continuous} homomorphisms) we can define the Pontryagin dual of a finitely generated $\Lambda$-module $X$ (such an $X$ carries a unique profinite topology for which the $\Lambda$-action is continuous by e.g. \cite[Proposition 5.2.23]{NSW}, namely the max-adic topology from the Noetherian local ring $\Lambda$); in this case $\widehat{X}$ is a discrete $\mathscr{O}$-torsion $\Lambda$-module, and Pontryagin duality sets up an anti-equivalence of categories between cofinitely generated discrete $\mathscr{O}$-torsion $\Lambda$-modules and finitely generated $\Lambda$-modules, with the usual double duality isomorphisms $\Lambda$-linear. As we will deal exclusively with cofinitely generated discrete $\mathscr{O}$-torsion $\Lambda$-modules and their Pontryagin duals (which are then finitely generated), we can essentially ignore topological considerations, e.g. continuity of $\Lambda$-module homomorphisms between modules of either of these types is automatic (this being apparent in the discrete case and a consequence of loc. cit. in the finitely generated case). Note also that a map between arbitrary Hausdorff topological $\Lambda$-modules is a $\Lambda$-module homomorphism if and only if it is an $\mathscr{O}[\Gamma]$-module homomorphism, i.e., $\mathscr{O}$-linear and $\Gamma$-equivariant. This follows from the density of the abstract group ring in $\Lambda$, and we will use it without comment in what follows. We begin by recalling the structure theorem for finitely generated $\Lambda$-modules. %good
\begin{thm}
\label{lambda-structure-thm}
If $X$ is a finitely generated $\Lambda$-module, then there is a pseudo-isomorphism
\begin{equation*}
X\to\Lambda^r\oplus\sum_{i=1}^t\Lambda/\pi^{m_i}\Lambda\oplus\sum_{j=1}^s\Lambda/\mathfrak{p}_j^{e_j}\text{,}
\end{equation*}
where $r,t,s\geq 0$ (with the convention that if one of these is zero, the summand doesn't appear), the $m_i$ and $e_j$ are positive integers, and the $\mathfrak{p}_j$ are (not necessarily distinct) height one prime ideals of $\Lambda$, all uniquely determined by $X$ (a pseudo-isomorphism is a $\Lambda$-module homomorphism with finite kernel and cokernel). 
\end{thm}
\begin{proof}
See \cite[5.3.8]{NSW}.
\end{proof}
Using the data of \cref{lambda-structure-thm}, we make the following definition of the Iwasawa invariants of $X$ and its dual.%good
\begin{defn}
\label{iwasawa-inv-defn}
If $X$ is a finitely generated $\Lambda$-module and we choose a pseudo-isomorphism as in \cref{lambda-structure-thm}, then the integer $r$ is the \emph{$\Lambda$-rank} of $X$, and the non-negative integers $\sum_{i=1}^t m_i$ and $\sum_{j=1}^s \rank_\mathscr{O}(\Lambda/\mathfrak{p}_j^{e_j})$ are the \emph{$\mu$-invariant} and the \emph{$\lambda$-invariant} of $X$, respectively. These are denoted $\rank_\Lambda(X)$, $\mu(X)$, and $\lambda(X)$. If $M=\widehat{X}$, we define the \emph{$\Lambda$-corank} of $M$ to be $\corank_\Lambda(M)=\rank_\Lambda(X)$, and the $\mu$ and $\lambda$-invariants of $M$ to be $\mu(M)=\mu(X)$, and $\lambda(M)=\lambda(X)$. We say that $M$ is \emph{$\Lambda$-cotorsion} if $X$ is $\Lambda$-torsion, i.e. if $\corank_\Lambda(M)=\rank_\Lambda(X)=0$.
\end{defn}%good
\indent For $n\geq 0$, let $\Gamma_n=\Gamma^{p^n}$, $G_n=\Gamma/\Gamma^{p^n}$, and for $m\geq 1$, let $\Lambda_{n,m}=(\mathscr{O}/\pi^m\mathscr{O})[G_n]$ and $\Lambda_n=\mathscr{O}[G_n]$. These are finitely generated $\Lambda$-modules, the former being discrete and $\mathscr{O}$-torsion. For fixed $m$, the $\Lambda_{n,m}$ form an inverse system and a directed system of $\Lambda$-modules. The transition maps for the inverse system are the natural restriction maps $\Res_{n,m}:\Lambda_{n+1,m}\rightarrow\Lambda_{n,m}$ induced by $\Res_{n,m}(g^\prime)=g^\prime\vert_{G_n}$ for $g^\prime\in G_{n+1}$, where $g^\prime\vert_{G_n}$ denotes the canonical image of $g^\prime$ in $G_n$. The transition maps for the directed system are the natural corestriction maps $\Cor_{n,m}:\Lambda_{n,m}\rightarrow\Lambda_{n+1,m}$ induced by $\Cor_{n,m}(g)=\sum_{g^\prime\in G_{n+1},\Res_{n,m}(g^\prime)=g}g^\prime$. Similarly the modules $\Lambda_n$ form an inverse system and a directed system via restriction maps $\Res_n:\Lambda_{n+1}\rightarrow\Lambda_n$ and corestriction maps $\Cor_n:\Lambda_n\rightarrow\Lambda_{n+1}$ defined by the same formulas. Of course,  $\Lambda=\varprojlim_n\Lambda_n$ by definition. Our first result concerns the self-duality of $\Lambda_{n,m}$. %good
\begin{prop}
\label{self-dual}
For $n\geq 0,m\geq 1$, there are canonical $\Lambda$-module isomorphisms 
\begin{equation*}
\varphi_{n,m}:\widehat{\Lambda_{n,m}}\rightarrow\Lambda_{n,m}
\end{equation*}
under which $\widehat{\Res_{n,m}}=\Cor_{n,m}$ and $\widehat{\Cor_{n,m}}=\Res_{n,m}$. More precisely, we have the equalities $\Cor_{n,m}\circ\varphi_{n,m}=\varphi_{n+1,m}\circ\widehat{\Res_{n,m}}$ and $\Res_{n,m}\circ\varphi_{n+1,m}=\varphi_{n,m}\circ\widehat{\Cor_{n,m}}$.
\end{prop}
\begin{proof}
The assertion is that there are isomorphisms $\varphi_{n,m}$ making the diagrams
\begin{equation*}
\xymatrix{\widehat{\Lambda_{n,m}}\ar[r]^{\widehat{\Res_{n,m}}} \ar[d]_{\varphi_{n,m}} & \widehat{\Lambda_{n+1,m}}\ar[d]^{\varphi_{n+1,m}}\\
\Lambda_{n,m}\ar[r]_{\Cor_{n,m}} & \Lambda_{n+1,m}}
\end{equation*}
and
\begin{equation*}
\xymatrix{\widehat{\Lambda_{n+1,m}}\ar[r]^{\widehat{\Cor_{n,m}}} \ar[d]_{\varphi_{n+1,m}} & \widehat{\Lambda_{n,m}}\ar[d]^{\varphi_{n,m}}\\
\Lambda_{n+1,m}\ar[r]_{\Res_{n,m}} & \Lambda_{n,m}}
\end{equation*}
commute. We have $\widehat{\Lambda_{n,m}}=\Hom_\mathscr{O}(\Lambda_{n,m},K/\mathscr{O})$. Since $\pi^m$ kills $\Lambda_{n,m}$, this is the same as
\begin{equation*}
\Hom_{\mathscr{O}/\pi^m\mathscr{O}}(\Lambda_{n,m},\pi^{-m}\mathscr{O}/\mathscr{O})\simeq\Hom_{\mathscr{O}/\pi^m\mathscr{O}}
(\Lambda_{n,m},\mathscr{O}/\pi^m)\simeq\Lambda_{n,m}\text{,}
\end{equation*}
where the first isomorphism comes from $[\pi^m]:\pi^{-m}\mathscr{O}/\mathscr{O}\rightarrow\mathscr{O}/\pi^m\mathscr{O}$ and the second isomorphism sends $\chi:\Lambda_{n,m}\rightarrow \mathscr{O}/\pi^m\mathscr{O}$ to $\sum_{g\in G_n}\chi(g) g$. We define $\varphi_{n,m}$ to be the composite of these isomorphisms. So, explicitly, $\varphi_{n,m}$ sends $\chi:\Lambda_{n,m}\rightarrow K/\mathscr{O}$ to $\sum_{g\in G_n}[\pi^m](\chi(g))g$. It is clear that $\varphi_{n,m}$ is an $\mathscr{O}$-module isomorphism, and we have, for $h\in G_n$,
\begin{align*}
\varphi_{n,m}(h\chi)&=\sum_{g\in G_n}[\pi^m]((h\chi)(g))g\\
&=\sum_{g\in G_n}[\pi^m](\chi(h^{-1}g))g\\
&=\sum_{g\in G_n}[\pi^m](\chi(g))(hg)\\
&=h\bigg(\sum_{g\in G_n}[\pi^m](\chi(g))g\bigg)=h\varphi_{n,m}(\chi)\text{.}
\end{align*}
Thus the map is $G_n$-equivariant, and therefore is a $\Lambda$-module isomorphism.\\
\indent Now for the diagrams, beginning with the first. Going horizontally then vertically sends $\chi\in\widehat{\Lambda_{n,m}}$ to $\sum_{g^\prime\in G_{n+1}}[\pi^m](\chi(\Res_{n,m}(g^\prime))) g^\prime$. If instead we go vertically first, $\chi$ is sent to $\sum_{g\in G_n}[\pi^m](\chi(g))g$, and then applying $\Cor_{n,m}$ gives 
\begin{align*}
\sum_{g\in G_n}[\pi^m](\chi(g))\Cor_{n,m}(g)&=
\sum_{g\in G_n}[\pi^m](\chi(g))\Bigg(\sum_{g^\prime\in G_{n+1},\Res_{n,m}(g^\prime)=g}g^\prime\Bigg)\\
&\sum_{g\in G_n}\Bigg(\sum_{g^\prime\in G_{n+1},\Res_{n,m}(g^\prime)=g}[\pi^m](\chi(\Res_{n,m}(g^\prime)))g^\prime\Bigg)\\
&=\sum_{g^\prime\in G_{n+1}}[\pi^m](\chi(\Res_{n,m}(g^\prime))) g^\prime\text{.}
\end{align*}
Thus the first diagram commutes.\\
\indent For the second diagram, given $\chi\in\widehat{\Lambda_{n+1,m}}$ and going horizontally, we get $\chi\circ\Cor_{n,m}$, and then going vertically gives
\begin{equation*}
\sum_{g\in G_n}[\pi^m](\chi(\Cor_{n,m}(g)))g\text{.}
\end{equation*}
Going vertically first gives $\sum_{g^\prime\in G_{n+1}}[\pi^m](\chi(g^\prime))g^\prime$, and then taking $\Res_{n,m}$, we obtain
\begin{equation*}
\sum_{g\in G_n}\Bigg(\sum_{g^\prime\in G_{n+1},\Res_{n,m}(g^\prime)=g}[\pi^m](\chi(g^\prime))\Bigg) g\text{.}
\end{equation*}
The definition of $\Cor_{n,m}$ gives
\begin{equation*}
\Cor_{n,m}(g)=\sum_{g^\prime\in G_{n+1},\Res_{n,m}(g^\prime)=g}g^\prime\text{,}
\end{equation*}
so $[\pi^m](\chi(\Cor_{n,m}(g)))=\sum_{g^\prime\in G_{n+1},\Res_{n,m}(g^\prime)=g}[\pi^m](\chi(g^\prime))$, which finishes the proof.
\end{proof}%good
\begin{cor}
\label{discrete-lim}
If $S_m=\varinjlim_n\Lambda_{n,m}$, with the limit taken with respect to the corestriction maps, then there is a canonical isomorphism of $\Lambda$-modules $\widehat{S_m}\simeq(\mathscr{O}/\pi^m\mathscr{O})[[\Gamma]]\simeq\Lambda/\pi^m\Lambda$. In particular, $S_m$ is a cofinitely generated, cotorsion $\Lambda$-module with $\lambda$-invariant zero and $\mu$-invariant $m$.
\end{cor}
\begin{proof}
As $S_m$ is a discrete $\mathscr{O}$-torsion $\Lambda$-module, by standard results on Pontryagin duality, $\widehat{S_m}$ is canonically isomorphic as a $\Lambda$-module to $\varprojlim_n\widehat{\Lambda_{n,m}}$, with the limit taken with respect to the maps $\widehat{\Cor_{n,m}}$. By \cref{self-dual}, this can be identified with the inverse limit of the $\Lambda_{n,m}$ taken with respect to the restriction maps, i.e., with $(\mathscr{O}/\pi^m\mathscr{O})[[\Gamma]]$. The second isomorphism is the inverse of the isomorphism given by passage to the quotient of the natural map $\Lambda\rightarrow(\mathscr{O}/\pi^m\mathscr{O})[[\Gamma]]$ (the kernel of the latter surjection is $\pi^m\Lambda$).
\end{proof}%good
\indent Now, instead of taking limits over $n\geq 0$, we want to take limits over $m\geq 1$. We will consider the discrete $\mathscr{O}$-torsion $\Lambda$-modules $S_n=\varinjlim_m(\mathscr{O}/\pi^m)[G_n]$, where the transition maps $[\pi]_{n,m}:\Lambda_{n,m}\rightarrow\Lambda_{n,m+1}$ are induced by the injective $\mathscr{O}$-module maps $\mathscr{O}/\pi^m\mathscr{O}\to\mathscr{O}/\pi^{m+1}\mathscr{O}$ given by multiplication by $\pi$.%good
\begin{prop}
\label{mult-by-pi-lim}
Under the isomorphisms $\varphi_{n,m}:\widehat{\Lambda_{n,m}}\simeq\Lambda_{n,m}$, $\widehat{[\pi]_{n,m}}=\theta_{n,m}$, where $\theta_{n,m}:\Lambda_{n,m+1}\rightarrow\Lambda_{n,m}$ is induced by the natural $\mathscr{O}$-module map $\beta_m:\mathscr{O}/\pi^{m+1}\mathscr{O}\rightarrow\mathscr{O}/\pi^m\mathscr{O}$. More precisely, $\theta_{n,m}\circ\varphi_{n,m+1}=\varphi_{n,m}\circ\widehat{[\pi]_{n,m}}$.
\end{prop}
\begin{proof}
The assertion is that the diagram
\begin{equation*}
\xymatrix{\widehat{\Lambda_{n,m+1}} \ar[r]^{\widehat{[\pi]_{n,m}}} \ar[d]_{\varphi_{n,m+1}} & \widehat{\Lambda_{n,m}}\ar[d]^{\varphi_{n,m}}\\
\Lambda_{n,m+1} \ar[r]_{\theta_{n,m}} & \Lambda_{n,m}}
\end{equation*}
commutes. Beginning with $\chi\in\widehat{\Lambda_{n,m+1}}$, going along the top horizontal map gives $\chi\circ[\pi]_{n,m}$, and then traveling vertically gives $\sum_{g\in G_n}[\pi^m](\chi([\pi]_{n,m}(g))) g$. Going the other way gives $\sum_{g\in G_n}\beta_m([\pi^{m+1}](\chi(g)))g$. To see that these coincide, fix $g\in G_n$, and let $r\in\pi^{-m-1}\mathscr{O}$ represent $\chi(g)\in\pi^{-m-1}\mathscr{O}/\mathscr{O}$. Then $\chi([\pi]_{n,m}(g))=\chi((\pi+\pi^{m+1}\mathscr{O})g)=\pi r+\mathscr{O}\in\pi^{-m}\mathscr{O}/\mathscr{O}$, so $[\pi^m](\chi([\pi]_{n,m}(g)))=\pi^{m+1}r+\pi^m\mathscr{O}$; since $[\pi^{m+1}](\chi(g))=\pi^{m+1}r+\pi^{m+1}\mathscr{O}$, this is exactly $\beta_m([\pi^{m+1}](\chi(g)))$. Thus the diagram commutes.
\end{proof}%good
\begin{cor}
\label{discrete-lim-coefficient}
There is a canonical $\Lambda$-module isomorphism $\widehat{S_n}\simeq\mathscr{O}[G_n]$.
\end{cor}
\begin{proof}
By \cref{mult-by-pi-lim} and Pontryagin duality, $\widehat{S_n}\simeq\varprojlim_m\Lambda_{n,m}$, where the limit is taken with respect to the maps given on coefficients by $\beta_m:\mathscr{O}/\pi^{m+1}\mathscr{O}\rightarrow\mathscr{O}/\pi^m\mathscr{O}$. These can be identified with the natural maps $\mathscr{O}[G_n]/(\pi^{m+1})\rightarrow\mathscr{O}[G_n]/(\pi^m)$, and taking the inverse limit of this system of modules produces $\mathscr{O}[G_n]$ because $\mathscr{O}[G_n]$ is a finite, hence $\pi$-adically complete $\mathscr{O}$-module.
\end{proof}%good
\indent Finally, we want to identify the Pontryagin dual of the discrete $\mathscr{O}$-torsion $\Lambda$-module 
\begin{equation*}
\mathscr{S}=\varinjlim_n S_n\text{,}
\end{equation*}
where the transition maps are $\psi_n=\varinjlim_m\Cor_{n,m}:\varinjlim_m\Lambda_{n,m}\rightarrow\varinjlim_m\Lambda_{n+1,m}$. This makes sense because the corestriction maps commute with the transition maps defining $S_n$ and $S_{n+1}$.%good
\begin{cor}
\label{double-lim}
There is a canonical $\Lambda$-module isomorphism $\widehat{\mathscr{S}}\simeq\Lambda$.
\end{cor}
\begin{proof}
As before, we have $\widehat{\mathscr{S}}\simeq\varprojlim_n\widehat{S_n}$, where the limit is taken with respect to the maps $\widehat{\psi_n}$. Under the isomorphism $\widehat{S_n}\simeq\varprojlim_m\Lambda_{n,m}$ coming from Pontryagin duality and \cref{self-dual}, the transition maps $\widehat{\psi_n}=\widehat{\varinjlim_m\Cor_{n,m}}$ for the modules $\widehat{S_n}$ become $\varprojlim_m\Res_{n,m}$ (again by \cref{self-dual}). That is, the diagram
\begin{equation*}
\xymatrix{\widehat{S_{n+1}}\ar[d]_\simeq  \ar[rr]^{\widehat{\psi_n}} & &\widehat{S_n}\ar[d]^\simeq\\
\varprojlim_m\Lambda_{n+1,m} \ar[rr]_{\varprojlim_m\Res_{n,m}} & &\varprojlim_m\Lambda_{n,m}}
\end{equation*}
commutes. As alluded to in the proof of \cref{discrete-lim-coefficient}, the inverse system consisting of the modules $\varprojlim_m\Lambda_{n,m}$ and the transition maps $\varprojlim_m\Res_{n,m}$ can be identified with the inverse system consisting of the modules $\mathscr{O}[G_n]$ and the transition maps $\Res_n:\mathscr{O}[G_{n+1}]\rightarrow\mathscr{O}[G_n]$. Thus $\widehat{\mathscr{S}}$ can be identified with $\varprojlim_n\mathscr{O}[G_n]=\Lambda$.
\end{proof}%mostly good
\begin{cor}
\label{tensor-module}
Let $M$ be a cofinitely generated $\mathscr{O}$-module, $M\simeq (K/\mathscr{O})^r\oplus\sum_{i=1}^t\mathscr{O}/\pi^{m_i}\mathscr{O}$. Then there is a $\Lambda$-module isomorphism $\varinjlim_n M\otimes_\mathscr{O}\mathscr{O}[G_n]\simeq\widehat{\Lambda}^r\oplus\sum_{i=1}^t\widehat{\Lambda/\pi^{m_i}\Lambda}$, where the limit in the source is taken with respect to the maps $\id_M\otimes\Cor_n$ and $\Lambda$ acts on the right tensor factor of each $M\otimes_\mathscr{O}\mathscr{O}[G_n]$. In particular, $\varinjlim_n M\otimes_\mathscr{O}\mathscr{O}[G_n]$ is a cofinitely generated $\Lambda$-module with corank $\corank_\mathscr{O}(M)$, $\lambda$-invariant zero, and $\mu$-invariant $\sum_{i=1}^tm_i$.
\end{cor}
\begin{proof}
For each $n\geq 0$, we have a canonical isomorphism of $\Lambda$-modules
\begin{equation*}
\big((K/\mathscr{O})^r\oplus\sum_{i=1}^t\mathscr{O}/\pi^{m_i}\mathscr{O}\big)\otimes_\mathscr{O}\mathscr{O}[G_n]
\simeq(K/\mathscr{O}\otimes_\mathscr{O}\mathscr{O}[G_n])^r\oplus\sum_{i=1}^t\Lambda_{n,m_i}\text{.}
\end{equation*}
These isomorphisms are compatible with the natural transition maps on both source and target as $n$ varies (all coming from corestriction), and since direct limits commute with $\otimes_\mathscr{O}$ and finite direct sums, in the limit over $n$ we obtain
\begin{equation*}
(\varinjlim_n K/\mathscr{O}\otimes\mathscr{O}[G_n])^r\oplus\sum_{i=1}^t\varinjlim_n(\Lambda_{n,m_i})\text{.}
\end{equation*}
For the factor on the right, Pontryagin duality together with \cref{discrete-lim} shows that the limit is $\sum_{i=1}^t\widehat{\Lambda/\pi^{m_i}\Lambda}$. For the left factor, we have, for each $n$,
\begin{equation*}
K/\mathscr{O}\otimes\mathscr{O}[G_n]\simeq\varinjlim_m\Lambda_{n,m}=S_n\text{,}
\end{equation*}
where we use that $K/\mathscr{O}=\varinjlim_m\pi^{-m}\mathscr{O}/\mathscr{O}\simeq\varinjlim\mathscr{O}/\pi^m\mathscr{O}$, the second limit over $m\geq 1$ taken via the injections induced by multiplication by $\pi$. As $n$ varies, under these identifications, the transition maps for the modules $K/\mathscr{O}\otimes_\mathscr{O}[G_n]$ coming from corestriction become the transition maps $\psi_n:S_n\rightarrow S_{n+1}$ used to define the module $\mathscr{S}$ of \cref{double-lim}. That corollary, along with another application of Pontryagin duality, shows that, upon taking the limit, we obtain $\widehat{\Lambda}$.
\end{proof}%mostly good
\section{Twisting of $\Lambda$-modules and characteristic polynomials}
\label{app-b}
\indent In this appendix we analyze the effect of twisting by a character on the characteristic polynomial of a finitely generated torsion $\Lambda$-module (see \cref{char-poly-defn} below). We retain the notation of \cref{app-a}. Let $q=p$ if $p$ is odd and $q=4$ if $p=2$, and let $\kappa:\Gamma\rightarrow 1+q\mathbf{Z}_p$ be a continuous character. Since the source and target of $\kappa$ are isomorphic to $\mathbf{Z}_p$, $\kappa$ is either trivial or injective, and in the latter case, it induces an isomorphism of $\Gamma$ onto its necessarily open image.\\%good except for reference fix
\indent Once we fix a topological generator $\gamma$ of $\Gamma$ and identify $\Lambda$ with $\mathscr{O}[[T]]$ via $\gamma\mapsto 1+T$ as in \cite[Proposition 5.3.5]{NSW}, we can associate to $\kappa$ a continuous $\mathscr{O}$-algebra endomorphism $\varphi_\kappa$ of $\Lambda$, determined uniquely by the requirement that $\varphi_\kappa(T)=\kappa(\gamma)(1+T)-1$. This is valid because $\kappa(\gamma)$ is a principal unit, and thus $\kappa(\gamma)(1+T)-1$ lies in the unique maximal ideal of $\Lambda$. The map $\varphi_\kappa$ is an automorphism because $\psi=\varphi_{\kappa^{-1}}$ satisfies $(\psi\circ\varphi)(T)=T=(\varphi\circ\psi)(T)$, and the only continuous $\mathscr{O}$-algebra endomorphism of $\Lambda$ fixing $T$ is the identity.\\%good
\indent Now, if $X$ is any $\Lambda$-module, we define a new $\Lambda$-module $X(\kappa)$ whose underlying $\mathscr{O}$-module is $X$, but with $\Lambda$-action twisted by $\varphi_\kappa$, i.e., we define $\lambda\cdot x=\varphi_\kappa(\lambda)x$ for $\lambda\in\Lambda$ and $x\in X$. We call $X(\kappa)$ the \emph{twist} of $X$ by $\kappa$. Since $\varphi_\kappa$ is an automorphism of $\Lambda$, it is clear that $X$ is finitely generated (respectively torsion) if and only if $X(\kappa)$ is. In particular, if $X$ is finitely generated and torsion, then so is $X(\kappa)$. In this case, we can consider the characteristic polynomials, in the sense of the following definition, for all twists of $X$.%good

\begin{defn}
\label{char-poly-defn}
Let $X$ be a finitely generated torsion $\Lambda$-module and set $X_K=X\otimes_\mathscr{O}K$, a finite-dimensional $K$-vector space with an action of $\mathscr{O}[[T]]$ via the chosen isomorphism $\Lambda\simeq\mathscr{O}[[T]]$. Then the \emph{characteristic polynomial} of $X$ (with respect to the chosen topological generator $\gamma$ inducing the $\mathscr{O}$-algebra isomorphism $\Lambda\simeq\mathscr{O}[[T]]$) is the polynomial in $\mathscr{O}[t]$ defined by
\begin{equation*}
\pi^{\mu(X)}\det((t\id_{X_K}-T)\vert X_K)\text{,}
\end{equation*}
except in the case that $X_K=0$ (i.e. $X=X[\pi^\infty]$), where we simply define the characteristic polynomial to be $\pi^{\mu(X)}$.
\end{defn}%good
\indent The characteristic polynomial of a finitely generated torsion $\Lambda$-module $X$ is of degree $\lambda(X)$ and is monic precisely when $\mu(X)=0$. We wish to describe the effect that twisting by $\kappa$ has on the characteristic polynomial, i.e., to give a formula for the characteristic polynomial for $X(\kappa)$ in terms of the characteristic polynomial of $X$. We will assume vanishing of the $\mu$-invariant as this is the only case needed for our desired application.%good
\begin{prop}
\label{char-poly-twist}
Let $X$ be a finitely generated torsion $\Lambda$-module. Let $F(t)\in\mathscr{O}[t]$ be the characteristic polynomial of $X$, and assume $\mu(X)=0$. Then $\mu(X(\kappa))=0$ and the characteristic polynomial of $X(\kappa)$ is 
\begin{equation*}
\kappa(\gamma)^{\lambda(X)}F(\kappa(\gamma)^{-1}(1+t)-1)\text{.}
\end{equation*}
\end{prop}
\begin{proof}
The vanishing of $\mu(X)$ is equivalent to finiteness of $X[\pi^\infty]$ (essentially by definition of the $\mu$-invariant). Since $X$ and $X(\kappa)$ have the same underlying $\mathscr{O}$-module, it follows that $\mu(X(\kappa))=0$ as well. If $X\otimes_\mathscr{O}K=0$, then $X(\kappa)\otimes_\mathscr{O} K=0$, and both characteristic polynomials are equal to $1$, which is consistent with the formula in the statement of the proposition. Assume then that $X\otimes_\mathscr{O}K\neq 0$ and let $x_1,\ldots,x_d\in X$ be elements whose images in $X/X[\pi^\infty]$ form an $\mathscr{O}$-basis (so $d=\rank_\mathscr{O}(X)=\lambda(X)$). Then $x_1\otimes 1,\ldots,x_d\otimes 1\in X\otimes_\mathscr{O}K$ form a $K$-basis, and because $\mu(X)=0$, if $[T]$ is the matrix for the endomorphism $T$ of $X\otimes_\mathscr{O}K$ with respect to the chosen basis, then $F(t)=\det(It-[T])$, where $I$ is the $d\times d$ identity matrix over $K$. Since $X(\kappa)\otimes_\mathscr{O}K$ has the same underlying $K$-vector space as $X\otimes_\mathscr{O}K$, the $x_i\otimes 1$ constitute a $K$-basis for this space as well. By definition, the action of $T$ on $X(\kappa)\otimes_\mathscr{O}K$ coincides with the action of $\kappa(\gamma)(1+T)-1$ on $X\otimes_\mathscr{O}K$. In other words, if $[\kappa(\gamma)(1+T)-1]$ is the matrix for the endomorphism $\kappa(\gamma)(1+T)-1$ acting on $X\otimes_\mathscr{O}K$ with respect to the chosen basis, then the characteristic polynomial of $X(\kappa)$ is $\det(It-[\kappa(\gamma)(1+T)-1])$. We have
\begin{align*}
It-[\kappa(\gamma)(1+T)-1]&=It-\kappa(\gamma)I-\kappa(\gamma)[T]+I\\
&=I(t-\kappa(\gamma)+1)-\kappa(\gamma)[T]\\
&=\kappa(\gamma)(I(\kappa(\gamma)^{-1}t-1+\kappa(\gamma)^{-1})-[T])\\
&=\kappa(\gamma)(I(\kappa(\gamma)^{-1}(1+t)-1)-[T])\text{.}
\end{align*}
Thus the characteristic polynomial of $X(\kappa)$ is 
\begin{equation*}
\det(\kappa(\gamma)(I(\kappa(\gamma)^{-1}(1+t)-1)-[T]))=\kappa(\gamma)^dF(\kappa(\gamma)^{-1}(1+t)-1)\text{,}
\end{equation*} 
where $d=\lambda(X)$, as claimed.
\end{proof}%good
\indent Continuing to assume that $X$ is finitely generated and torsion with characteristic polynomial $F(t)$, but not necessarily with $\mu$-invariant zero, one can prove using the structure theory for such $\Lambda$-modules that $X_{\Gamma_n}$ (the module of $\Gamma_n$-coinvariants of $X$) is finite if and only if $F(t)$ has no zeros in $\overline{K}$ of the form $\zeta-1$, where $\zeta$ is a $p^n$-th root of unity. We use this observation together with our calculation of the characteristic polynomial of a twist in \cref{char-poly-twist} to establish the next result.%good
\begin{prop}
\label{finite-inv}
Assume $\kappa$ is non-trivial. Then for all but finitely many $i\in\mathbf{Z}$, $X(\kappa^i)_{\Gamma_n}$ is finite for all $n\geq 0$.
\end{prop}
\begin{proof}
To begin, it follows from the structure theorem for finitely generated $\Lambda$-modules that for all $i\in\mathbf{Z}$ and $n\geq 0$, $X(\kappa^i)_{\Gamma_n}$ is finite if and only if $(X(\kappa^i)/X(\kappa^i)[\pi^\infty])_{\Gamma_n}$ is finite. Combining this fact with the exactness of twisting, we see that it is no loss of generality to assume that $\mu(X)=0$ (so $\mu(X(\kappa^i))=0$ for all $i$). Let $F_i(t)=\kappa(\gamma)^{i\lambda(X)}F(\kappa(\gamma)^{-i}(1+t)-1)$, where, again, $F(t)$ is the characteristic polynomial of $X$. By \cref{char-poly-twist}, this is the characteristic polynomial of $X(\kappa^i)$. Suppose that $i,j\in\mathbf{Z}$ are distinct integers for which there exist $n_1,n_2\geq 0$ with $X(\kappa^i)_{\Gamma_{n_1}}$ and $X(\kappa^j)_{\Gamma_{n_2}}$ infinite. Then there are $\zeta\in\mu_{p^{n_1}}(\overline{K})$ and $\zeta^\prime\in\mu_{p^{n_2}}(\overline{K})$ such that $F_i(\zeta-1)=0=F_j(\zeta^\prime-1)$. This means that $\kappa(\gamma)^{-i}\zeta-1$ and $\kappa(\gamma)^{-j}\zeta^\prime-1$ are roots of $F$. If these roots are the same, then $\kappa(\gamma)^{-i}\zeta=\kappa(\gamma)^{-j}\zeta^\prime$, whence $\kappa(\gamma)^{j-i}=\zeta^{-1}\zeta^\prime$. The right-hand side of this last equation is visibly a root of unity, but the left-hand side is an element of $1+q\mathbf{Z}_p$, which is torsion-free. Thus $\kappa(\gamma)^{j-i}=1$, which forces $i=j$ (because $\kappa$ is non-trivial), contrary to assumption. So the roots of $F$ are distinct. It follows that we may define an injective map from the set of integers $i\in\mathbf{Z}$ for which $X(\kappa^i)_{\Gamma_n}$ is not finite for \emph{some} $n$ to the set of roots of $F(t)$, which is finite as $F(t)\neq 0$. The former set is therefore finite as well.
\end{proof}%good

\end{document}